\documentclass[11pt,a4paper]{article}
\usepackage[utf8]{inputenc}
\usepackage[a4paper, total={6.6in, 9.6in}]{geometry}

\usepackage[T1]{fontenc}
\usepackage[variablett]{lmodern}
\usepackage{amsmath}
\usepackage{amsfonts}
\usepackage{amssymb}
\usepackage{mathtools}
\usepackage{bbm}
\usepackage{float}
\usepackage{wrapfig}
\usepackage{enumerate}
\usepackage{enumitem}
\usepackage[svgnames]{xcolor}
\usepackage{graphicx}
\usepackage{quiver}
\usepackage{array}

\usepackage{caption}
\usepackage[bottom]{footmisc}

\usepackage[hidelinks,hypertexnames=false]{hyperref}
\definecolor{mycolor}{HTML}{750000}
\hypersetup{
    colorlinks = true,
    allcolors = mycolor,
    bookmarksdepth = 2,
}
\PassOptionsToPackage{unicode}{hyperref}
\PassOptionsToPackage{naturalnames}{hyperref}

\usepackage[nameinlink]{cleveref}
\Crefname{subsection}{Subsection}{Subsections}
\Crefname{question}{Question}{Questions}
\Crefname{subsubsection}{Paragraph}{Paragraphs}

\usepackage{comment}

\newcommand{\customref}[2]{\hyperref[#2]{#1}}

\newcommand\iref[2]{\customref{\Cref*{#1}~\ref*{#2}}{#1}}

\newenvironment*{lsubstack}{%
  \setlength\arraycolsep{0pt}%
  \begin{array}{>{\scriptstyle}l}}{\end{array}}

\renewcommand{\epsilon}{\varepsilon}
\renewcommand{\phi}{\varphi}

\renewcommand{\hat}{\widehat}
\renewcommand{\tilde}{\widetilde}

\newcommand{\N}{\mathbb{N}}
\newcommand{\Z}{\mathbb{Z}}

\newcommand{\Q}{\mathbb{Q}}
\newcommand{\C}{\mathbb{C}}
\newcommand{\R}{\mathbb{R}}

\newcommand{\lcm}{\mathrm{lcm}} 

\newcommand{\Pm}{\mathcal{P}}
\newcommand{\Pex}{\Pm^{\mathrm{ex}}}

\newcommand{\Res}{\mathrm{Res}} 

\newcommand{\tr}{\mathrm{tr}} 
\newcommand{\End}{\mathrm{End}} 
\newcommand{\Aut}{\mathrm{Aut}} 
\newcommand{\op}{\mathrm{op}} 
\newcommand{\Cent}{\mathrm{Cent}} 
\newcommand{\Br}{\mathrm{Br}} 
\newcommand{\Gal}{\mathrm{Gal}} 
\newcommand{\Inn}{\mathrm{Inn}} 
\newcommand{\Out}{\mathrm{Out}} 
\newcommand{\Sym}{\mathfrak{S}} 
\newcommand{\card}[1]{\left| #1 \right|} 
\newcommand{\ord}{\mathrm{ord}} 

\newcommand{\Frob}{\mathrm{Frob}} 

\newcommand{\verti}{\, \middle \vert \,}

\newcommand{\Nm}[1]{|\!|#1|\!|} 

\newcommand{\ind}{\mathrm{ind}} 

\newcommand{\Mu}{\frac M u}

\renewcommand{\O}{\mathcal{O}} 
\newcommand{\e}{\mathrm{e}} 

\newcommand{\inv}[1]{\mathrm{inv}\leftl(#1\rightr)} 

\newcommand{\into}{\hookrightarrow}
\newcommand{\simto}{\overset\sim\to}

\newcommand{\avg}{\mathrm{avg}} 
\newcommand{\cycgcd}{\mathrm{cycgcd}}

\newcommand{\ram}{\mathrm{ram}}

\newcommand{\charindex}{k}
\def\psichi{{\langle\psi, \chi\rangle}}

\newcommand\otspam{\mathrel{\reflectbox{\ensuremath{\mapsto}}}}

\newcommand{\leftl}{\mathopen{}\mathclose\bgroup\left}
\newcommand{\rightr}{\aftergroup\egroup\right}

\usepackage{amsthm}
\usepackage{thmtools}

\newcounter{mycounter}[section]

\theoremstyle{plain}
\newtheorem{theorem}[mycounter]{Theorem}
\newtheorem{corollary}[mycounter]{Corollary}
\newtheorem{proposition}[mycounter]{Proposition}
\newtheorem{lemma}[mycounter]{Lemma}

\theoremstyle{remark}
\newtheorem{remark}[mycounter]{Remark}

\theoremstyle{definition}
\newtheorem{definition}[mycounter]{Definition}
\newtheorem{question}[mycounter]{Question}

\counterwithin{equation}{section}

\usepackage{titletoc}
\titlecontents{section}
[0pt]                                               
{}%
{\contentsmargin{0pt}                               
    \large\thecontentslabel.\enspace
    }
{\contentsmargin{0pt}\large}                        
{\titlerule*[.5pc]{ }\contentspage}                 
[]                                                  

\usepackage{titlesec}
\titleformat{\section}[block]{\normalfont\centering\scshape\large}{\thesection.}{1em}{}
\titleformat{\subsection}[block]{\normalfont\large\bf}{\thesubsection.}{1em}{\bf}
\titleformat{\subsubsection}[runin]{\normalfont}{\bf\thesubsubsection.}{0.3em}{\bf}

\makeatletter
\patchcmd{\@maketitle}{\LARGE}{\huge}{\typeout{OK 1}}{\typeout{Failed 1}}
\patchcmd{\@maketitle}{\large \lineskip}{\Large \lineskip}{\typeout{OK 2}}{\typeout{Failed 2}}
\makeatother

\title{
  Asymptotics of extensions of simple $\Q$-algebras
}

\author{
  Fabian Gundlach%
  \footnote{
    Universität Paderborn, Fakultät EIM, Institut für Mathematik, Warburger Str. 100, 33098 Paderborn, Germany.
  }
  \footnote{
    Email: \texttt{fabian.gundlach@uni-paderborn.de}.
  }
  \and
  Béranger Seguin%
  $^*$%
  \footnote{
    Email: \texttt{math@beranger-seguin.fr}.
  }
}

\renewenvironment{abstract}{%
\hfill\begin{minipage}{0.95\textwidth}
\rule{\textwidth}{1pt} \textsc{Abstract.}}
{\par\noindent\rule{\textwidth}{1pt}\end{minipage}}

\begin{document}

\maketitle{}

\begin{abstract}
  We answer various questions concerning the distribution of extensions of a given central simple algebra $K$ over a number field.
  Specifically, we give asymptotics for the count of inner Galois extensions $L/K$ of fixed degree and center with bounded discriminant.
  We also relate the distribution of outer extensions of~$K$ to the distribution of field extensions of its center~$Z(K)$.
  This paper generalizes the study of asymptotics of field extensions to the noncommutative case in an analogous manner to the program initiated by Deschamps and Legrand to extend inverse Galois theory to division algebras.
  
  \bigskip

\end{abstract}

{
  \hypersetup{linkcolor=black}
  \tableofcontents{}
}
\hfill\rule{0.95\textwidth}{1pt}

\section{Introduction}

\subsection{Context}

The study of statistics of field extensions turns inverse Galois theory into a quantitative problem, replacing the question of the existence of extensions with a given Galois group by the question of their asymptotic distribution.
The main conjecture in this area was introduced by Malle in \cite{malle1,malle2} as a proposed generalization of results of Mäki and Wright for abelian extensions of number fields \cite{maki,wright}.
This conjecture has received a lot of attention and has been studied using various methods: for some recent articles, see for example
\cite{
  wang-symmetric-times-abelian,%
  klnilp,%
  etw,%
  eszb,%
  koymans-pagano,%
  motte%
}.
Another active area concerns the distribution of extensions of fixed degree without specifying a Galois group, see
\cite{
  cohncube,%
  davenport,%
  bharga-quart,%
  bharga-quint%
} for small degrees and
\cite{
  schmidt,%
  ev,%
  couvnf,%
  bsw,%
  lot%
} for all degrees.

Noncommutative Galois theory was developed in \cite{jac40,jac47,cartan47}.
We will define the notions we need, but readers wanting to learn more may refer to \cite{cohn,jacobson} or to the introductory sections of \cite{deschamps-petites,deschamps}.
In \cite{desleg}, a program was initiated%
\footnote{
  One may argue that first steps towards this program were taken by research concerning \emph{admissible} groups, cf. for example \cite{schacher,harbater}.
}
to extend inverse Galois theory to division~rings.
Since then, this question has been vastly explored, and fundamental results were obtained in the articles \cite{alp,beh21,deschamps-beh,bdl,leg22b,leg22,leg24}.
This article aims, in a similar manner, to study the quantitative aspects of extensions of noncommutative algebras.

\subsection{Focus of this work}

Our objects of study are (finite-dimensional) simple $\Q$-algebras.
If $K$ is such an algebra, its center~$Z(K)$ is a number field.

\begin{definition}
  \label{def:extension}
  An \emph{extension} of a simple $\Q$-algebra $K$ is a simple $\Q$-algebra~$L$ equipped with an injective $\Q$-algebra homomorphism $K \into L$.
  We usually think of $K$ as a subalgebra of $L$ via this embedding.
  An \emph{isomorphism} between extensions $L_1$ and $L_2$ of $K$ with embeddings $e_1:K \into L_1$ and $e_2:K \into L_2$ is a $\Q$-algebra isomorphism $i:L_1 \simto L_2$ such that $e_2 = i \circ e_1$.
\end{definition}

Let $L/K$ be an extension of simple $\Q$-algebras.
We denote by $\Aut(L/K)$ the automorphism group of $L/K$, i.e., the set of $\Q$-algebra automorphisms of $L$ which act trivially on $K$.
The \emph{degree} is defined as $[L:K] \coloneqq \frac{\dim_{\Q}(L)}{\dim_{\Q}(K)}$.
If $K$ is a division algebra, the degree agrees with the dimension of $L$ both as a left $K$-vector space and as a right $K$-vector space.
In \Cref{ssn:discriminant}, we define a quantity $d(L/K)\in\Q_{>0}$, which we treat as ``(the absolute value of) the norm of the relative discriminant of~$L/K$''.
One may then ask the following question:
\begin{question}
  \label{qn:main}
  Let $K$ be a simple $\Q$-algebra and $n \geq 2$.
  How many isomorphism classes of extensions~$L/K$ of degree $n$ are there which satisfy the bound $d(L/K) \leq X$, asymptotically as $X \to +\infty$?
\end{question}

\Cref{qn:main} is very general.
For instance, the extensions \Cref{qn:main} aims to count include field extensions of number fields, which are notoriously hard to parametrize.
Instead of studying this general problem, we address two more specific questions, focusing only on certain types of extensions.
More precisely, we study the asymptotics of ``inner Galois extensions'' and of ``outer extensions'', defined below.
The former question turns out to admit a complete answer which we give in \Cref{thm:main-inner}, whereas the latter reduces to well-studied questions concerning the distribution of commutative field extensions, as we explain in \Cref{sn:outer}.
These two types of extensions are representative of all extensions of $K$, as by \Cref{thm:general-bijection} every extension $L/K$ of simple algebras splits ``naturally'' into a tower $L/L'/K$ where $L/L'$ is inner Galois and $L'/K$ is outer.
(Here, $L'$ is the double-centralizer of $K$ in $L$.)
This fact could be used to address more general variants of \Cref{qn:main}, as we briefly discuss in \Cref{sn:gen-ext}.

Throughout the article, we make a special effort to include simple algebras which are not division algebras in all discussions, but we also systematically prove the corresponding statements focusing exclusively on division algebras.

\subsubsection{Inner Galois extensions.}

In \Cref{sn:inner}, we restrict our attention to \emph{inner Galois extensions} of a simple $\Q$-algebra $K$, where an extension $L/K$ is:
\begin{itemize}
  \item
    \emph{inner} if all of its automorphisms are inner, i.e., given by conjugation by an element of $L^\times$;
  \item
    \emph{Galois} if $K$ equals the algebra $L^{\Aut(L/K)} \coloneqq \left\{ x \in L \,\big\vert\, \forall \sigma \in \Aut(L/K), \sigma(x) = x\right\}$.
\end{itemize}
As we explain in \Cref{lem:inner-galext-incl-center}, an extension $L/K$ is inner Galois if and only if $Z(L)\subseteq Z(K)$.

In \Cref{thm:main-inner}, we give asymptotics for the number of inner Galois extensions $L/K$ with given degree $n$ and given center $Z=Z(L)$ that satisfy the discriminant bound $d(L/Z)\leq X$.
These asymptotics take the form
$
  C X^{1/a} (\log X)^{b-1}
$
for explicitly given constants $a$ and $b$, and for a real number $C$ which is positive unless no inner Galois extension of $K$ of degree $n$ with center $Z$ exists.

Fixing the center lets us reduce the problem to a question about central simple $Z$-algebras.
The count of all inner Galois extensions of $K$ of degree $n$ (with any center) can in principle be obtained by summing the resulting asymptotics over the finitely many subfields $Z$ of $Z(K)$.

In the case $K=Z(K)=Z$, the extensions we are counting are exactly the central simple $Z$-algebras of degree $n$ satisfying the discriminant bound $d(L/Z)\leq X$.
This special case of \Cref{thm:main-inner} was already established in \cite[Theorem~1.5 and Lemma~3.2]{lrpt}.

Previous work on this question also includes \cite[Corollary~4]{fks}, where (combined with  \cite[\href{https://stacks.math.columbia.edu/tag/074Z}{Theorem~074Z}]{stacks-project}) the authors prove that infinitely many central simple $Z$-algebras $L$ contain a given commutative field extension $K=Z(K)$ of $Z$ as a maximal subfield (i.e., $[L:Z] = [K:Z]^2$).
The definition of the main exponent $1/a$ in our final asymptotics relies on a group-theoretical lemma they prove to this end (\Cref{lem:FKS}).

Our proof strategy is similar to that of \cite{lrpt}: we reduce the problem to counting central simple $Z$-algebras satisfying certain local conditions, we parametrize these algebras by elements of the Brauer group $\Br(Z)$, and we set up a Dirichlet series counting them.
The local-global principle for Brauer groups lets us write the Dirichlet series as a sum of finitely many Euler products.
Finally, we determine its rightmost ``pole'' by comparison with Artin L-functions and apply a Tauberian theorem to prove \Cref{thm:main-inner}.
Additional computations ensure that the leading coefficients of our asymptotics are positive when there is at least one such extension.
We also give the asymptotics when~$K$ is a division algebra and we count only extensions which are division algebras.

We also prove \Cref{thm:main-inner-prod-ram}, which is a variant of \Cref{thm:main-inner} in which discriminants are replaced by products of ramified primes.
The proof strategy is identical.

\subsubsection{Outer extensions.}

An extension $L$ of $K$ is \emph{outer} if it has no nontrivial inner automorphisms.
In \Cref{sn:outer}, we prove \Cref{thm:charac-outer}, which shows that outer extensions $L/K$ are exactly those of the form $L = F'\otimes_{Z(K)}K$ for a field extension $F'$ of $Z(K)$.
This generalizes a theorem of Deschamps and Legrand \cite[Corollaire~2]{desleg}.
A consequence of this theorem is that the problem of counting outer extensions of $K$ reduces to the notoriously difficult problem of counting field extensions of $Z(K)$.
Additional computations let us characterize extensions which are division algebras and compute discriminants in terms of invariants of the extension $F'/Z(K)$ (\Cref{thm:tensor-disc-divalg}).
We give a few applications of these ideas in \Cref{subsn:counting-outer}.

\subsubsection{General extensions.}

In \Cref{sn:gen-ext}, we briefly discuss the possibility of adapting the methods of \Cref{sn:inner,sn:outer} in order to count more general extensions $L/K$ by decomposing them into inner and outer extensions.
We highlight a few helpful facts, but also analytic difficulties specific to this problem.

\subsection{Preliminaries and notation}

\subsubsection{Brauer groups of local and global fields.}
\label{sssn:brauer-refresher}

\paragraph{Central simple algebras.}
In this article, simple algebras are systematically assumed to be finite-dimensional over their center.
If $F$ is a field and $K$ is a central simple $F$-algebra, we denote by $[K] \in \Br(F)$ the class of~$K$ in the Brauer group of $F$.
The \emph{index} of $K$ is the integer $\ind(K)$ such that~$K$ is isomorphic to a matrix algebra over a central division $F$-algebra of dimension $\ind(K)^2$.
Note that $K$ is a division algebra if and only if $[K:F] = \ind(K)^2$.
The \emph{exponent} of $K$ is the order of $[K]$ in $\Br(F)$.
When $F$ is a global or local field, which is systematically true in this article, the exponent of $K$ equals its index \cite[(31.4),~(32.19)]{reiner} (see \cite{deschamps-indices} for counterexamples in the general case).

\paragraph{Brauer groups of local fields.}

Brauer groups of local fields admit explicit descriptions:
\begin{itemize}
  \item
    The Brauer group of $\C$ is trivial.
    We identify it with the trivial subgroup of $\Q/\Z$ via the trivial group homomorphism $\mathrm{inv} : \Br(\C) \to \Q/\Z$.
  \item
    The Brauer group of $\R$ is isomorphic to $\Z/2\Z$, generated by the class of the algebra of Hamilton quaternions over~$\R$.
    We identify it with the subgroup $\{0,\frac 1 2\}$ of $\Q/\Z$ via the group homomorphism $\mathrm{inv} : \Br(\R) \to \Q/\Z$ mapping the nontrivial element to $\frac 1 2$.
  \item
    If $F$ is a non-archimedean local field, then there is an isomorphism $\mathrm{inv} : \Br(F) \overset\sim\to \Q/\Z$ \cite[(31.8)]{reiner}.
\end{itemize}

The \emph{Hasse invariant} of a central simple algebra over a local field is the image in $\Q/\Z$ of its class.

\paragraph{Brauer groups of global fields.}

Assume $F$ is a global field.
If $K$ is a central simple $F$-algebra and~$v$ is a place of $F$, we denote by $F_v$ the completion of $F$ at $v$ and by $K_v \coloneqq K \otimes_{F} F_v$ the completion of~$K$ at $v$, which is a central simple algebra over the local field $F_v$.
We call the element $\inv{K_v} \in \Q/\Z$ the \emph{local invariant} of $K$ at $v$.

Let $\Pm$ be the set of all places of $F$.
The local-global principle for Brauer groups of global fields (the Albert--Brauer--Hasse--Noether theorem) is summed up by the exact sequence \cite[(32.13)]{reiner}:
\begin{equation}
  \label{ses-brauer}
  1
  \to
  \Br(F)
  \to
  \bigoplus_{v \in \Pm}
    \Br(F_v)
  \underset{{\scriptstyle\sum_v}\mathrm{inv_v}}{\longrightarrow}
  \Q/\Z
  \to
  1
\end{equation}
where ${\scriptstyle\sum_v}\mathrm{inv_v}$ is the sum in $\Q/\Z$ of the Hasse invariants of the coordinates of an element of
$
  \bigoplus_{v \in \Pm}
      \Br(F_v)
$.
In particular, a central simple $F$-algebra is uniquely determined (up to isomorphism) by its dimension $M^2$ and by the collection of its local invariants in $\Q/\Z$ (indexed by places of $F$), on which the only constraints are the following:
\begin{itemize}
  \item
    all local invariants have order dividing $M$ in $\Q/\Z$;
  \item
    all but finitely many local invariants are trivial;
  \item
    the local invariants at complex places are trivial;
  \item
    the local invariants at real places are either trivial or equal to $\frac 1 2$;
  \item
    the sum of the local invariants over all places of $F$ is trivial.
\end{itemize}


\subsubsection{Terminology and notation.}
\label{sssn:terminology}

If $S$ is a finite set, we denote by $\card{S}$ its cardinality.
We denote by $\e : \C \to \C$ the function $z \mapsto \exp(2 \pi i z)$.
If $R$ is a ring, we denote by $\mathfrak{M}_n(R)$ the algebra of $n \times n$ matrices over~$R$ and by $\End_R(A)$ the algebra of endomorphisms of a (left or right) $R$-module $A$.
We denote by $\Nm{p}$ the norm of a prime $p$ of a number field $F$.

For $n \in \N$, we see $\Z/n\Z$ as a subgroup of $\Q/\Z$, namely that of elements whose order divides~$n$: if $a \in \Z/n\Z$, we denote by $\frac a n$ the corresponding element of $\Q/\Z$.
When speaking about elements of $\Z/n\Z$, the phrases ``$a$ divides $b$'', ``$a$ is the greatest common divisor (resp. the least common multiple) of $b$ and $c$'' and ``$b$ and $c$ are coprime'' must be interpreted as statements about the corresponding principal ideals, identified with positive divisors of $n$.
For instance, the greatest common divisor of~$b$ and $c$ is the unique positive divisor of~$n$ generating the same ideal of $\Z/n\Z$ as~$b$ and $c$ together, i.e., $\gcd\leftl(\tilde b, \tilde c, n\rightr)$ where $\tilde b, \tilde c \in \N$ are arbitrary representatives of $b,c$.

If $L$ is a $\Q$-algebra and $K$ is a subalgebra of $L$, we use the following notation:
\begin{itemize}
  \item
    The \emph{centralizer} $\Cent_L(K)$ is the subalgebra of $L$ consisting of those elements that commute with all elements of $K$;
  \item
    $\Inn(L/K)$ is the normal subgroup of $\Aut(L/K)$ consisting of inner automorphisms, i.e., those corresponding to conjugation by an element of $\Cent_L(K)^{\times}$.
\end{itemize}
Two elements of $\Cent_L(K)^{\times}$ induce the same inner automorphism if and only if they differ by an element of the center of $L$.
Therefore, $\Inn(L/K) \simeq \Cent_L(K)^{\times}/Z(L)^{\times}$.

\subsection{Discriminants}
\label{ssn:discriminant}

We associate to an extension $L/K$ of simple $\Q$-algebras a number $d(L/K)$, which we use as a substitute for the absolute value of the norm of the relative discriminant of $L$ over~$K$.
When $L/K$ is an extension of number fields, $d(L/K)$ has precisely that meaning.

\subsubsection{The case $K \subseteq Z(L)$.}
\label{par:discr-centalg}

Assume $K$ is contained in the center of $L$.
In this case, there is a well-defined notion of discriminant: the number field $K$ has a unique maximal $\Z$-order, namely its ring of integers $\O_K$, and one can choose a maximal $\O_K$-order $\Lambda$ in $L$ \cite[(10.4)]{reiner}.
Although $\Lambda$ is not unique in general, the discriminant of $\Lambda/\O_K$ does not depend on the choice of $\Lambda$ \cite[(25.3)]{reiner}.
Therefore, we simply denote by $d(L/K)$ the integer obtained as the absolute value of the norm of the discriminant of $\Lambda/\O_K$, for any choice of maximal order $\Lambda$ in $L$.

First, consider the case $K = Z(L)$.
Then, $L$ is a central simple $K$-algebra of some dimension~$m^2$.
For each prime $p$ of $K$, let $\lambda_p \in \Z/m\Z$ be such that $\inv{L_p} = \frac{\lambda_p}{m}$.
The local index $m_p = \ind(L_p)$ is the reduced denominator of this fraction, i.e., $m_p = \frac{m}{\gcd(m, \, \lambda_p)}$.
By comparing dimensions, we see that~$L_p$ is an algebra of $\kappa_p \times \kappa_p$-matrices over a central division $K_p$-algebra of dimension $m_p^2$ for $\kappa_p = \frac{m}{m_p}$.
A formula for the norm of the relative discriminant of $L/Z(L)$ follows directly from \cite[(25.10)]{reiner}:
\begin{equation*}
  d\big(L/Z(L)\big)
  =
  \left(
    \prod_p
    \Nm{p}^{(m_p-1)\kappa_p}
  \right)^m
\end{equation*}
where the product is taken over primes $p$ of $Z(L)$.
This can be rewritten in the following ways:
\begin{equation}
  \label{eqn:discr-csa}
  d\big(L/Z(L)\big)
  =
  \prod_p
  \Nm{p}^{m^2\left(1-\frac 1 {m_p}\right)}
  =
  \prod_p
  \Nm{p}^{m \big( m- \gcd(m,\, \lambda_p) \big)}.
\end{equation}

When $K$ is a subfield of $Z(L)$, the computation of~$d(L/K)$ reduces to the central case using the following ``relative discriminant formula'', which is a special case of \cite[Exercise 25.1a]{reiner}:
\begin{equation}
  \label{eqn:reldisc-form-center}
    d(L/K)
    =
    d\big(
      L/Z(L)
    \big)
    \cdot
    d\big(
      Z(L)/Z(K)
    \big)^{[L:Z(L)]}.
\end{equation}

\subsubsection{A general definition.}

To measure the ``size'' of a general extension of simple $\Q$-algebras, we use the following quantity, which is both natural (cf. \Cref{prop:reldisc}) and mysterious (cf.~\Cref{rk:disc-notint}):

\begin{definition}
  \label{def:discriminant}
  Let $L/K$ be an extension of simple $\Q$-algebras.
  We denote by $d(L/K)$ the following positive rational number:
  \[
    d(L/K)
    =
    \frac{
      d(L/\Q)
    }{
      d(K/\Q)^{[L:K]}
    }.
  \]
\end{definition}

\begin{proposition}
  \label{prop:reldisc}
  \Cref{def:discriminant} is the only possible definition of a height $d(L/K)$ which coincides with the norm of the relative discriminant when $K \subseteq Z(L)$, and which satisfies the relative discriminant formula $d(M/K) = d(M/L) d(L/K)^{[M:L]}$ for every tower of extensions $M/L/K$.
\end{proposition}

\begin{proof}
  The uniqueness follows from the case $K=\Q$ of the relative discriminant formula for an arbitrary extension $M/L$.
  The relative discriminant formula for a tower $M/L/K$ of extensions follows formally from the ``usual'' relative discriminant formula for commutative fields combined with \Cref{eqn:reldisc-form-center}.
\end{proof}

\begin{remark}
  \label{rk:disc-notint}
  The number $d(L/K)$ is in general not an integer.
  For example, no prime but $2$ is ramified in the $\Q$-algebra $L = \Q(i,j,k)$ of Hamilton quaternions, so $d(L/\Q)$ is a power of $2$, but $L$ contains the commutative subfield $K = \Q(i+j+k) \simeq \Q(\sqrt{-3})$ in which $3$ is ramified, so $3 \mid d(K/\Q)$.
  Hence, the denominator of $d(L/K)$ is divisible by $3$.
\end{remark}

\subsection{Acknowledgements}

This work was supported by the Deutsche Forschungsgemeinschaft (DFG, German Research Foundation) --- Project-ID 491392403 --- TRR 358 (project A4).
The authors thank Bruno Deschamps, Markus Kirschmer, and the anonymous referee for helpful discussions and feedback.

\section{Inner Galois extensions}
\label{sn:inner}

In this section, we state and prove \Cref{thm:main-inner}, which gives asymptotics for the distribution of inner Galois extensions of a given (finite-dimensional) simple $\Q$-algebra, of fixed degree and center.

\Cref{ssn:facts-inngal} contains useful lemmas concerning inner Galois extensions.
In \Cref{ssn:notation-and-mainthm}, we fix some notation and state the main theorem, \Cref{thm:main-inner}.

The proof of \Cref{thm:main-inner} is divided into \Cref{ssn:reduction-combi,ssn:key-dirichlet,ssn:analytic-props,ssn:positivity-leading-coeff,ssn:restriction-divalg}: in \Cref{ssn:reduction-combi}, we rephrase the problem in combinatorial terms; in \Cref{ssn:key-dirichlet}, we set up the Dirichlet series for this counting problem; in \Cref{ssn:analytic-props}, we describe analytic properties of the Dirichlet series and apply a Tauberian theorem; in \Cref{ssn:positivity-leading-coeff}, we check that the leading coefficient in our estimates is positive under the assumption that an extension exists (this proves the main statement \iref{thm:main-inner}{thm:main-inner-i}); finally, in \Cref{ssn:restriction-divalg}, we establish \iref{thm:main-inner}{thm:main-inner-ii}, which is the result when one excludes extensions which are not division algebras.

In \Cref{subsn:prod-ram}, we explain how to adapt the proof in order to prove \Cref{thm:main-inner-prod-ram}, a variant of \Cref{thm:main-inner} where the height by which we count is the product of ramified primes.
Finally, in \Cref{subsn:criteria-existence}, we give criteria for the existence of an extension as in \Cref{thm:main-inner}.

\subsection{General facts about inner Galois extensions}
\label{ssn:facts-inngal}

In this subsection, we establish general properties of inner Galois extensions.
We begin with a characterization (\Cref{lem:inner-galext-incl-center}), explain why inner Galois extensions of a simple algebra~$K$ with center~$Z$ can be identified with central simple $Z$-algebras in which $K$ embeds (\Cref{prop:you-can-forget-the-embedding}), and prove a criterion to decide whether there is an embedding between two simple algebras (\Cref{lem:csa-inclusion-criterion}).

We first prove the two following lemmas, which are not specific to inner Galois extensions:

\begin{lemma}
  \label{lem:non-eigenvalue}
  Let $K$ be a (finite-dimensional) algebra over an infinite field $F$.
  For every $x \in K$, there is a $\lambda \in F$ such that $x - \lambda$ is invertible.
\end{lemma}

\begin{proof}
  Embedding $K$ into an algebra of matrices over $F$ lets one see $x$ as a square matrix with coefficients in $F$.
  Since its characteristic polynomial has finitely many roots and $F$ is infinite, there is an element $\lambda \in F$ such that $x - \lambda$ is invertible.
\end{proof}

\begin{lemma}
  \label{lem:fixed-under-inn-centcent}
  Let $L/K$ be an extension of simple $\Q$-algebras.
  Then $L^{\Inn(L/K)} = \Cent_L(\Cent_L(K))$.
\end{lemma}

\begin{proof}
  Elements of $\Inn(L/K)$ are given by conjugation by elements of $\Cent_L(K)^\times$.
  Thus, $L^{\Inn(L/K)} = \Cent_L(\Cent_L(K)^\times)$.
  In particular, $\Cent_L(\Cent_L(K)) \subseteq L^{\Inn(L/K)}$.
  Conversely, if $x \in L^{\Inn(L/K)}$ and $y \in \Cent_L(K)$, use \Cref{lem:non-eigenvalue} to pick a $\lambda \in \Q$ such that $y - \lambda I \in \Cent_L(K)^\times$; since $x$ belongs to $L^{\Inn(L/K)} = \Cent_L(\Cent_L(K)^{\times})$, it commutes with $y - \lambda I$ and thus with~$y$.
\end{proof}

The following lemma characterizes inner Galois extensions in a simple manner:

\begin{lemma}
  \label{lem:inner-galext-incl-center}
  An extension $L/K$ of simple $\Q$-algebras is inner Galois if and only if $Z(L) \subseteq Z(K)$.
\end{lemma}

\begin{proof}
  ~
  \begin{description}
    \item[($\Rightarrow$)]
      Since $L/K$ is inner Galois, we have
      $
        K
        = L^{\Inn(L/K)}
        \overset{\customref{Lem.\ref*{lem:fixed-under-inn-centcent}}{lem:fixed-under-inn-centcent}}{=} \Cent_L(\Cent_L(K))
      $.
      Hence:
      \[
        Z(L)
        \subseteq
        \Cent_L(K) \cap \Cent_L(\Cent_L(K))
        = \Cent_L(K) \cap K
        = Z(K).
      \]
    \item[($\Leftarrow$)]
      By the Skolem--Noether theorem, the extension $L/Z(L)$ is inner.
      Since $K$ contains $Z(K)$ and thus $Z(L)$, this implies that $L/K$ is also inner.
      Proving that $L/K$ is Galois then amounts to proving that the $Z(L)$-algebra $K$ equals $L^{\Inn(L/K)}$, which is $\Cent_L(\Cent_L(K))$ by \Cref{lem:fixed-under-inn-centcent}.
      The equality $K = \Cent_L(\Cent_L(K))$ follows from the centralizer theorem \cite[\href{https://stacks.math.columbia.edu/tag/074T}{Theorem~074T}]{stacks-project}.
      \qedhere
  \end{description}
\end{proof}

\begin{proposition}
  \label{prop:you-can-forget-the-embedding}
  The map sending an isomorphism class of inner Galois extensions $L/K$ with center $Z$ (as defined in \Cref{def:extension}) to the isomorphism class of $L$ as a $Z$-algebra (forgetting about the embedding $K \into L$) is injective.
\end{proposition}

\begin{proof}
  Let $L_1$, $L_2$ be central simple $Z$-algebras, isomorphic via an isomorphism $i : L_1 \overset\sim\to L_2$
  , and in which $K$ embeds respectively via embeddings $e_1$ and $e_2$.
  By the form of the Skolem--Noether theorem given in \cite[\href{https://stacks.math.columbia.edu/tag/074Q}{Theorem~074Q}]{stacks-project}, there is an (inner) automorphism $\alpha$ of~$L_2$ such that $\alpha\circ i\circ e_1 = e_2$.
  Hence, $(L_1,e_1)$ and $(L_2,e_2)$ are isomorphic extensions of $K$ in the sense of \Cref{def:extension}.
\end{proof}

A consequence of \Cref{prop:you-can-forget-the-embedding} is that the problem of counting inner Galois extensions of~$K$ with center $Z$ can be equivalently rephrased as counting central simple $Z$-algebras in which $K$ embeds.
This rephrasing is especially useful when combined with the following criterion, which lets one decide whether a simple $Z$-algebra $K$ embeds into a central simple $Z$-algebra~$L$:

\begin{lemma}
  \label{lem:csa-inclusion-criterion}
  Let $F/Z$ be a field extension of degree $d$.
  Let $L$ be a central simple $Z$-algebra of dimension $M^2$ and $K$ be a central simple $F$-algebra of dimension $m^2$.
  The following are equivalent:
  \begin{enumerate}[label=(\roman*)]
    \item
      \label{csa-inclusion-criterion-i}
      There is an embedding $K \into L$ of $Z$-algebras.
    \item
      \label{csa-inclusion-criterion-ii}
      The number $j \coloneqq \frac{M}{dm}$ is an integer, and there is a central simple $F$-algebra $R$ of dimension $j^2$ such that $[L\otimes_Z F] = [K] + [R]$ in the Brauer group of $F$.
  \end{enumerate}
\end{lemma}

The degree of $L$ over $K$ is then $n \coloneqq dj^2$.
The situation is summed up by the following diagram:
\[
  \begin{tikzcd}[column sep=1ex, row sep=3ex]
    L \ar[dashed,-]{rd}{n=dj^2} \ar[-,swap]{ddd}{M^2=(dmj)^2} \\
    & K \ar[-]{d}{m^2} \\[1ex]
    & F \ar[-]{ld}{d} \\
    Z
  \end{tikzcd}
\]


\begin{proof}[Proof of \Cref{lem:csa-inclusion-criterion}.]
  ~
  \begin{description}
    \item[
      \ref{csa-inclusion-criterion-i}
      $\Rightarrow$
      \ref{csa-inclusion-criterion-ii}
    ]
      Assume $K$ embeds in $L$ and see $K$ as a subring of $L$ via this embedding.
      Let $R = \Cent_L(K)$.
      By the centralizer theorem \cite[\href{https://stacks.math.columbia.edu/tag/074T}{Theorem~074T}]{stacks-project},~$R$ is a simple $Z$-algebra of dimension $\frac {M^2} {dm^2}$ whose centralizer $\Cent_L(R)$ is $K$.
      In particular:
      \[
        Z(R)
        =
        R \cap \Cent_L(R)
        =
        \Cent_L(K) \cap K
        =
        Z(K)
        =
        F.
      \]
      So $R$ is a central simple $F$-algebra of dimension $\frac{M^2}{d^2 m^2} = j^2$.
      In particular, $j$ is an integer.
      See~$L$ as a right $(K^{\op}\otimes_Z L)$-module via the action induced by the formula $\lambda.(a \otimes \lambda') = a\lambda\lambda'$ for $\lambda, \lambda'~\in~L$ and $a \in K$.
      An endomorphism $\phi$ of the right $(K^{\op}\otimes_Z L)$-module $L$ is determined by the element $\phi(1)$.
      This lets us identify $\End_{K^{\op} \otimes_Z L}(L)$ with a subset of $L$.
      We let the reader check that this subset is precisely $\Cent_L(K) = R$.
      By \cite[\href{https://stacks.math.columbia.edu/tag/074F}{Lemma~074F}]{stacks-project}, the $Z$-algebra $K^{\op} \otimes_Z L$ is simple because both~$K$ and~$L$ are simple and $Z(L)=Z$.
      By \cite[\href{https://stacks.math.columbia.edu/tag/074E}{Lemma~074E~(5)}]{stacks-project}, the equality $\End_{K^{\op} \otimes_Z L}(L) = R$ then implies that $\End_R(L) = K^{\op} \otimes_{Z} L = K^{\op} \otimes_F \big( F \otimes_Z L\big)$.
      Moreover, $\End_R(L)$ is a matrix algebra over~$R$ by \cite[\href{https://stacks.math.columbia.edu/tag/074E}{Lemma~074E~(6)}]{stacks-project}.
      Therefore, the classes of $K^{\op} \otimes_F \big( F \otimes_Z L\big)$ and of $R$ coincide in the Brauer group of $F$, which implies $[F\otimes_Z L] -[K] = [R]$ and finally \ref{csa-inclusion-criterion-ii}.

    \item[
      \ref{csa-inclusion-criterion-ii}
      $\Rightarrow$
      \ref{csa-inclusion-criterion-i}
    ]
      Let $K' \coloneqq K\otimes_F R$.
      By assumption, we have $[L\otimes_Z F]=[K']$ in the Brauer group of $F$.
      Since $\dim_F(L\otimes_Z F) = \dim_Z(L) = M^2 = (dmj)^2$ and $\dim_F(K') = \dim_F(K)\cdot\dim_F(R) = (mj)^2$, this implies that $L\otimes_Z F \simeq \mathfrak{M}_d(K')$.

      Consider the central simple $Z$-algebra
      $
        \End_Z(K')
        \simeq
        \mathfrak{M}_{\dim_Z(K')}(Z)
      $.
      There are embeddings $K'\into \End_Z(K')$ and $(K')^{\op}\into\End_Z(K')$ coming from the respective actions of $K'$ on itself via left and right multiplication.
      The images of these two embeddings commute as $K'$ is associative.

      Let $A = \End_Z(K') \otimes_Z \mathfrak{M}_d(Z)$.
      The embedding $K'\into \End_Z(K')$ induces the following embedding of $L$ in $A$:
      \[
        L
        \into
        L \otimes_Z F
        \simeq
        \mathfrak{M}_d(K')
        \simeq
        K' \otimes_Z \mathfrak{M}_d(Z)
        \into
        \End_Z(K') \otimes_Z \mathfrak{M}_d(Z)
        =
        A.
      \]
      We see $L$ as a subalgebra of $A$ via this embedding.
      Let $C \coloneqq \Cent_A(L)$.
      As $L$ and~$A$ are central simple $Z$-algebras, we have $L\otimes_Z C \simeq A$ by \cite[\href{https://stacks.math.columbia.edu/tag/074U}{Lemma~074U}]{stacks-project}.
      Since $A\simeq\mathfrak{M}_{\dim_Z(K')\cdot d}(Z)$, it follows that $[L]=[C^{\op}]$ in the Brauer group of $Z$.
      We have $\dim_Z(L) = M^2 = (dmj)^2$ and:
      \[
        \dim_Z(C^{\op})
        =
        \dim_Z(C)
        =
        \frac
        {\dim_Z(A)}
        {\dim_Z(L)}
        =
        \frac{
          \dim_Z(K')^2\cdot d^2
        }
        {
          (dmj)^2
        }
        =
        \frac{
          (dm^2j^2)^2\cdot d^2
        }
        {
          (dmj)^2
        }
        =
        (dmj)^2.
      \]
      Therefore, there is an isomorphism $L\simeq C^{\op}$.

      The inclusions $L \into K' \otimes_Z \mathfrak{M}_d(Z) \into A$ imply that $C = \Cent_A(L)$ contains $\Cent_A(K'\otimes_Z\mathfrak{M}_d(Z))$.
      Elements in the image of the embedding $(K')^{\op} \into \End_Z(K')\into A$ commute with elements of $K'\otimes_Z\mathfrak{M}_d(Z)$ because they come from right multiplication by elements of $K'$.
      Therefore, these elements belong to $C$.
      This defines an embedding $(K')^{\op} \into C$, from which we finally obtain an embedding $K \into K\otimes_F R = K' = \big((K')^{\op}\big)^{\op} \into C^{\op} \simeq L$ as claimed.
      \qedhere
  \end{description}
\end{proof}

\begin{remark}
  Several cases of \Cref{lem:csa-inclusion-criterion} are classical:
  \begin{itemize}
    \item
      A commutative field extension $F$ of $Z$ of degree $M$ is contained in a central simple $Z$-algebra~$L$ of dimension $M^2$ (as a maximal subfield) if and only if it is a splitting field, i.e., $L\otimes_ZF\simeq\mathfrak{M}_M(F)$.
      This is the case $m=j=1$.
      Our proof of \Cref{lem:csa-inclusion-criterion} draws inspiration from the proof of this special case given in \cite[\href{https://stacks.math.columbia.edu/tag/074Z}{Theorem~074Z}]{stacks-project}.
    \item
      Two central simple $Z$-algebras $L$ and $K$ of the same dimension $M^2$ are isomorphic if and only if $[L]=[K]$ in the Brauer group of $Z$.
      This is the case $d=j=1$.
    \item
      When $F = Z$ (i.e., $d = 1$), \Cref{lem:csa-inclusion-criterion} specializes to a criterion for the inclusion of a central simple $Z$-algebra into another.
      This criterion appears in \cite[Section~5]{deschamps} (cf. the definition and description of what Deschamps calls the Brauer group $\Br(K)$ of a central simple $Z$-algebra~$K$):

      \begin{corollary}[Deschamps]
        \label{cor:descent-inner}
        Let $Z$ be a field, let $L$ be a central simple $Z$-algebra of dimension~$M^2$ containing a central simple $Z$-algebra $K$ of dimension $m^2$.
        Then, there is a central simple $Z$-algebra $R$ of dimension $\left( \frac M m \right)^2$ such that $L \simeq R \otimes_Z K$, namely $R = \Cent_L(K)$.
      \end{corollary}
  \end{itemize}
\end{remark}

\subsection{Notation and main theorem}
\label{ssn:notation-and-mainthm}

In this subsection, after introducing the necessary terminology and notation, we state our main theorem, \Cref{thm:main-inner}.
The notations we fix here are in effect throughout all of \Cref{sn:inner}.

\subsubsection{The centers.}
We fix an extension $F/Z$ of number fields and we let $d = [F:Z]$.
We denote the set of places of $Z$ by $\Pm$.
For every place~$w$ of $F$, lying above a place $v$ of $Z$, the \emph{local degree of~$F/Z$ at $w$} is the integer $d_w \coloneqq [F_w : Z_v]$.

We let $G$ be the Galois group of the Galois closure $\hat F$ of $F/Z$.
The transitive action of $G$ on the~$d$ embeddings of $F$ into $\hat F$ lets us see $G$ as a transitive subgroup of $\Sym_d$.
For an unramified prime~$p$ of~$Z$, we let $\Frob(p)$ be the conjugacy class of $G$ consisting of the Frobenius automorphisms for primes of~$\hat F$ above~$p$.

For an element $g \in G$, we denote by $\cycgcd(g)$ the greatest common divisor of the sizes of all the orbits of the action of $g$ on $\{1, \dots, d\}$.
Note that $\cycgcd(g)$ divides $d$, the sum of the sizes of all orbits.
Since $\cycgcd(g)$ only depends on the conjugacy class of $g$,
we use the same notation when $g$ is a conjugacy class of $G$.
Finally, we let:
\[
  U
  \,\coloneqq\,
  \underset{g \in G}{\lcm}
    \,\,
    \cycgcd(g)
  .
\]

\begin{lemma}
  \label{lem:FKS}
  We have $U \geq 2$, unless $F = Z$.
\end{lemma}

\begin{proof}
  Assume that $F \neq Z$, i.e., $d \geq 2$.
  By a theorem of Fein, Kantor and Schacher%
  \footnote{%
    Thanks to Michael Giudici for pointing this theorem to us.
    Note that the proof of Fein, Kantor and Schacher relies on the classification of finite simple groups.
  }
  \cite[Theorem 1]{fks}, the transitive subgroup $G$ of $\Sym_d$ contains a fixed-point-free element $g$ whose order is a prime power $p^k$.
  Its orbits all have sizes divisible by $p$, so $p \mid \cycgcd(g) \mid U$.
\end{proof}

We describe $U$ explicitly in two special cases:

\begin{lemma}\label{lem:u0-galois}
  If $F/Z$ is Galois, then $\cycgcd(g) = \ord(g)$ for all $g \in G$, and $U$ is the exponent of~$G$.
\end{lemma}

\begin{proof}
  We have $d = \card{G}$ and $G \into \Sym_d$ is the regular embedding.
  The orbits of $g\in G$ all have size~$\ord(g)$ and thus $\cycgcd(g) = \ord(g)$.
  Finally,
  $
    U
    =
    \lcm_{g\in G}
      \,
      \ord(g)
  $
  is the exponent of~$G$.
\end{proof}

\begin{lemma}\label{lem:u0-primepower}
  If $d=p^k$ is a prime power with $k\geq1$, then $U=p^{k'}$ for some $1\leq k'\leq k$.
\end{lemma}

\begin{proof}
  By \Cref{lem:FKS}, we have $U\geq2$.
  On the other hand, $U$ is by definition a divisor of $d$.
\end{proof}

\subsubsection{The central simple algebra.}
We fix a central simple $F$-algebra $K$ of dimension $m^2$.
For every place $w$ of $F$, we denote by~$\kappa_w$ the element of $\Z/m\Z$ such that the local invariant $\inv{K_w} \in \Q/\Z$ of~$K$ at $w$ is $\frac{\kappa_w}{m}$.

\begin{definition}
  \label{def:pex}
  We say that a place $v$ of $Z$ is \emph{exceptional} if it is archimedean, or ramified in~$F$, or if $\kappa_w \neq 0$ for some place $w|v$ of $F$.
  We denote by $\Pex$ the finite set of exceptional places of~$Z$.
\end{definition}

\subsubsection{The degree.}
We fix an integer $j \geq 1$.
We let $n = dj^2$ and $M = dmj$.
In the rest of \Cref{sn:inner}, our goal is to count inner Galois extensions $L/K$ of degree $n$ with center $Z(L) = Z$.
If $n=1$, then $d=j=1$ and there is exactly one such extension, namely the trivial extension $L=K$.
From now on, we exclude this case and assume $n \geq 2$.
Since $n = dj^2 \geq 2$, we have $j \geq 2$ or $d \geq 2$.
By \Cref{lem:FKS}, it follows that $U j \geq 2$.
Hence, the following definition makes sense:

\begin{definition}
  \label{def:u}
  We denote by $u$ the smallest prime number dividing $U j$.
\end{definition}

\begin{definition}
  \label{def:beta}
  We define the rational number $\beta \in (0,1]$ as follows:
  \[
    \beta
    \coloneqq
    \frac 1 {\card{G}}
    \cdot
    \Big|
      \Big\{
        g \in G
        \,\Big\vert\,
        u \text{ divides } j \cdot \cycgcd(g)
      \Big\}
    \Big|
    =
    \left\lbrace
      \begin{array}{cl}
        1
        &
        \text{if $u|j$},
        \\
        \frac 1 {\card{G}}
        \cdot
        \Big|
          \Big\{
            g \in G
            \,\Big\vert\,
            u \text{ divides } \cycgcd(g)
          \Big\}
        \Big|
        &
        \text{otherwise}.
      \end{array}
    \right.
  \]
\end{definition}

Note that $u$ divides $j \cdot \cycgcd(g) = \frac{M} {dm/\cycgcd(g)}$ if and only if $\frac {dm} {\cycgcd(g)}$ divides $\Mu$.
Hence:
\begin{equation}
  \label{eqn:beta-otherdef}
  \beta
  =
  \frac 1 {\card{G}}
  \cdot
  \card{
    \left\{
      g\in G
      \verti
      \frac {d m} {\cycgcd(g)} \textnormal{ divides }\Mu
    \right\}
  }.
\end{equation}

The following remarks help understand the constants $u$ and $\beta$:

\begin{remark}
  If $F/Z$ is a Galois extension or $d$ is a prime power, then $u$ is the smallest prime factor of~$dj$ by \Cref{lem:u0-galois,lem:u0-primepower}.
  Moreover, if $F/Z$ is Galois and $u$ does not divide $j$, then $\beta$ is the proportion of elements of $G$ whose order is divisible by $u$.
\end{remark}

\begin{remark}
  In the non-Galois case, the number $u$ is not necessarily the smallest prime factor of~$dj$.
  For instance, take $j=1$, $Z=\Q$, and any number field $F$ of degree $6$ whose Galois closure has Galois group
  $
    \Big\langle
      (1\;4)(2\;5),(1\;3\;5)(2\;4\;6)
    \Big\rangle
    \subseteq
    \Sym_6
  $, which is the transitive permutation group \texttt{6T4} in GAP notation and is isomorphic to $A_4$.
  This group contains no permutations whose cycles all have even sizes, i.e., $u \neq 2$.
  Instead, we have $u=3$ as there are elements consisting of two $3$-cycles.
\end{remark}

\subsubsection{Statement of the main theorem.}

Using the notation introduced above, we state the main result of this section:

\begin{theorem}
  \label{thm:main-inner}
  ~
  \begin{enumerate}[label=(\roman*)]
    \item
      \label{thm:main-inner-i}
      There is a real number $C \geq 0$ such that the number $N(X)$ of inner Galois extensions $L/K$ of degree $n=dj^2$ with center $Z$ and with $d(L/Z) \leq X$ satisfies
      \[
        N(X)
        \underset{X\to\infty}\sim
        C
        X^{1/a}
        (\log X)^{b-1}
      \]
      where $a = M^2\left(1-\frac1u\right)$ and $b=(u-1)\beta$.
      When $C = 0$, this is taken to mean that there is no such extension for any $X$.
    \item
      \label{thm:main-inner-ii}
      The same holds if we restrict to inner Galois extensions $L/K$ which are division algebras (with a possibly smaller constant $C$).
  \end{enumerate}
\end{theorem}

\begin{remark}
The relative discriminants $d(L/K)$ and $d(L/Z)$ differ by a constant factor that only depends on $K$ and $Z$ (cf. \Cref{ssn:discriminant}):
\[
  d(L/Z) = d(K/Z)^{n} \cdot d(L/K).
\]
Hence, \Cref{thm:main-inner} continues to hold, with a different constant $C$, if we replace the condition $d(L/Z)\leq X$ by $d(L/K)\leq X$.
\end{remark}

Proving \Cref{thm:main-inner} is the focus of \Cref{ssn:reduction-combi,ssn:key-dirichlet,ssn:analytic-props,ssn:positivity-leading-coeff,ssn:restriction-divalg}.
In \Cref{subsn:criteria-existence}, we give additional criteria to check the existence of an extension, in order to determine whether the leading coefficient~$C$ in \Cref{thm:main-inner} is positive.

\begin{remark}
  We obtain a finer version of \Cref{thm:main-inner} where we constrain the behavior of $L$ at finitely many places.
  Let $S$ be a finite set of places of $Z$ and let $\xi : S \to \Z/M\Z$ be a map.
  Then, \Cref{thm:main-inner} holds (with possibly smaller constants $C$) if one restricts to extensions whose local invariants at the places $v \in S$ are given by $\frac{\xi(v)}M$.
\end{remark}

\begin{remark}
  Our methods yield expressions for the leading coefficient $C$ in \Cref{thm:main-inner}, see \Cref{eqn:expr-constant} and \Cref{eqn:formula-tildef-at1}.
  These expressions involve an infinite product over all primes of~$Z$ and the values at $s=1$ (resp. the residue, for the trivial character) of the Artin $L$-functions of the irreducible characters of $G = \Gal(\hat F/Z)$ (cf. the proof of \Cref{lem:approx}).
  In \Cref{rk:constant-computation-when-psi-constant}, we remark that in certain cases, including the case $F=Z$, only the residue of the Dedekind zeta function of $Z$ at $1$ (which is given by the class number formula) is needed.
\end{remark}

\subsubsection{Counting by the product of ramified primes.}

In \cite{wood-probabilities-of-local-behaviors}, Wood popularized the question of counting number fields not by discriminant, but by the product of ramified primes, which in her language is a \emph{fair counting function} for abelian extensions.
For any simple $\Q$-algebra $L$ with center $Z$, we define:
\[
  \ram(L)
  =
  \prod_{\substack{
    p\textnormal{ prime of }Z\\
    \textnormal{ramified in }L}
  }
    \Nm{p}.
\]
In \Cref{subsn:prod-ram}, we explain how to adapt the proof of \Cref{thm:main-inner} to count inner Galois extensions by the product of their ramified primes.
This leads to the following result:

\begin{theorem}
  \label{thm:main-inner-prod-ram}
  ~
  \begin{enumerate}[label=(\roman*)]
    \item
      \label{thm:main-inner-prod-ram-i}
      There is a real number $C \geq 0$ such that the number $N(X)$ of inner Galois extensions $L/K$ of degree $n=dj^2$ with center $Z(L)=Z$ and with $\ram(L) \leq X$ satisfies
      \[
        N(X)
        \underset{X\to\infty}\sim
        C
        X
        (\log X)^{b^*-1}
      \]
      where
      $
        b^*
        =
        j
        \left(
          \frac{1}{\card{G}}
          \sum_{g\in G}
            \cycgcd(g)
        \right)
        -
        1
      $.
      When $C = 0$, this is taken to mean that there is no such extension for any $X$.
    \item
      \label{thm:main-inner-prod-ram-ii}
      The same holds if we restrict to inner Galois extensions $L/K$ that are division algebras (with a possibly smaller constant $C$).
  \end{enumerate}
\end{theorem}

\begin{remark}
  \label{rmk:bstar-positive}
  We have $b^* > 0$ because we assumed that $j \geq 2$ or $d \geq 2$, which by \Cref{lem:FKS} implies $\cycgcd(g)\geq2$ for some $g\in G$.
\end{remark}

\subsection{Combinatorial formulation of the counting problem}
\label{ssn:reduction-combi}

In this subsection, we rephrase the counting problem combinatorially with the help of the Albert--Brauer--Hasse--Noether theorem.
First, we specialize the criterion from \Cref{lem:csa-inclusion-criterion} to the case of number fields:

\begin{lemma}
  \label{lem:csa-inclusion-criterion-nf}
  Let $L$ be a central simple $Z$-algebra of dimension $M^2$.
  For every place $v$ of $Z$, let $\lambda_v$ be the element of $\Z/M\Z$ such that $\inv{L_v}=\frac{\lambda_v}{M}$.
  Then, the following are equivalent:
  \begin{enumerate}[label=(\roman*)]
    \item
      \label{csa-inclusion-criterion-nf-i}
      There is an embedding $K\into L$ of $Z$-algebras.
    \item
      \label{csa-inclusion-criterion-nf-ii}
      For each place~$w$ of $F$, lying above a place $v$ of $Z$, we have $dm \mid d_w \lambda_v - d j \kappa_w$ in $\Z/M\Z$.
  \end{enumerate}
\end{lemma}

\begin{proof}
  By \Cref{lem:csa-inclusion-criterion}, $K$ embeds in $L$ if and only if there is a central simple $F$-algebra $R$ of dimension~$j^2$ such that $[L \otimes_Z F] = [K] + [R]$ in $\Br(F)$.
  This amounts to the condition that the index of $[L \otimes_Z F] - [K]$ divide $j$.
  Since index and exponent coincide, and by the exact sequence of~\Cref{ses-brauer}, this means that for every place $w$ of $F$, we have
  $
    j
    \cdot
    ([L \otimes_Z F_w] - [K_w]) = 0
  $ in $\Br(F_w)$.
  By \cite[(31.9)]{reiner}, if $w$ is place of $F$ lying above a place~$v$ of $Z$, then $\inv{L \otimes_Z F_w} = [F_w:Z_v] \cdot \inv{L_v} = d_w \frac{\lambda_v} M$.
  Recall also that $\inv{K_w} = \frac {\kappa_w} m$.
  We finally obtain that $K$ embeds in $L$ if and only if, for each place $w$ of $F$, lying above a place~$v$ of~$Z$, the following equality holds in~$\Q/\Z$:
  \begin{align*}
    0
    & =
    j
    \cdot
    \Big(
      d_w \frac {\lambda_v} M
      -
      \frac {\kappa_w} m
    \Big)
    \\
    &
    =
    \frac{d_w\lambda_v}{dm} - \frac{j\kappa_w}{m}.
  \end{align*}
  This equality amounts to the divisibility $dm \mid d_w\lambda_v - dj\kappa_w$ in $\Z/M\Z$.
\end{proof}

We now give a combinatorial description of inner Galois extensions of $K$ of degree $n$ with center~$Z$:

\begin{definition}
  \label{def:Lambda}
  Let $\Lambda$ be the set of maps $\lambda: \Pm \rightarrow \Z/M\Z$ with finite support (i.e., identically zero outside a finite set) satisfying the following conditions:
  \begin{enumerate}[label=(\Roman*)]
    \item
      \label{def:Lambda:complex}
      for all complex places $v \in \Pm$, we have $\lambda(v)=0$.
    \item
      \label{def:Lambda:real}
      for all real places $v \in \Pm$, we have $\lambda(v)\in\{0,\frac M2\}$ (necessarily, $\lambda(v)=0$ if $M$ is odd).
    \item
      \label{def:Lambda:div}
      for all places $v \in \Pm$ and all places $w|v$ of $F$, we have $dm \mid d_w\lambda(v) - dj\kappa_w$ in $\Z/M\Z$.
    \item
      \label{def:Lambda:sum}
      $\sum_{v \in \Pm} \lambda(v) = 0$ in $\Z/M\Z$.
  \end{enumerate}
  Let $\Lambda'\subseteq\Lambda$ be the set of maps that additionally satisfy:
  \begin{enumerate}[label=(V)] 
    \item
      \label{def:Lambda:gcd}
      $\gcd_v \lambda(v) = 1$ in $\Z/M\Z$.
  \end{enumerate}
\end{definition}

\begin{theorem}
  \label{thm:ext-vs-map}
  There is a bijection between the set of isomorphism classes of inner Galois extensions $L/K$ of degree $n=dj^2$ with center $Z(L)=Z$ and the set $\Lambda$.
  Let $L$ be such an extension and $\lambda \in \Lambda$ be the corresponding map.
  Then, $L$ is a division algebra if and only if $\lambda$ lies in $\Lambda'$.
  Moreover, the norm $d(L/Z)$ of the relative discriminant of $L$ over its center $Z$ equals the following quantity~$d(\lambda)$, computed in terms of the map $\lambda$ alone:
  \[
    d(\lambda)
    \coloneqq
    \prod_{p\textnormal{ prime of }Z}
      \Nm{p}^{
        M
        \Bigl(
          M
          -
          \gcd\bigl(
            M, \, \lambda(p)
          \bigr)
        \Bigr)
      }.
  \]
\end{theorem}

\begin{proof}
  By \Cref{prop:you-can-forget-the-embedding}, isomorphism classes of inner Galois extensions of $K$ of degree $n$ with center $Z$ are in bijection with equivalence classes of central simple $Z$-algebras of dimension $M^2$ in which~$K$ embeds.
  By the characterizations of \Cref{sssn:brauer-refresher}, specifying a central simple $Z$-algebra~$L$ of dimension~$M^2$ is the same as giving a map $\lambda : \Pm \to \Z/M\Z$ with finite support satisfying conditions \ref{def:Lambda:complex}, \ref{def:Lambda:real} and \ref{def:Lambda:sum} of \Cref{def:Lambda}.
  The local invariant of $L$ at a place $v$ of $Z$ is then given by $\frac{\lambda(v)}{M} \in \Q/\Z$.
  By \Cref{lem:csa-inclusion-criterion-nf}, the existence of an embedding of $K$ into $L$ is equivalent to condition \ref{def:Lambda:div}.

  The central simple algebra $L$ is a division algebra if and only if it has index $M$.
  Since index and exponent coincide for central simple algebras over number fields, this is equivalent to the condition that the invariants $\inv{L_v}$ have least common denominator $M$, which is in turn equivalent to \ref{def:Lambda:gcd}.

  The formula for the discriminant follows from \Cref{par:discr-centalg}.
\end{proof}

For non-exceptional places, the condition \ref{def:Lambda:div} of \Cref{def:Lambda} takes a much simpler form:

\begin{lemma}
  \label{lem:csa-inclusion-unramified}
  Let $v \in \Pm \setminus \Pex$.
  Then, condition \ref{def:Lambda:div} of \Cref{def:Lambda} holds for all $w|v$ if and only if $\frac {d m} {\cycgcd(\Frob(v))}$ divides $\lambda(v)$.
\end{lemma}

\begin{proof}
  Since the prime $v$ is not exceptional, we have $\kappa_w = 0$ for primes $w|v$ of $F$.
  Thus, condition~\ref{def:Lambda:div} amounts to $\lambda(v)$ being a multiple of $\frac{dm}{\gcd(dm,d_w)}$ for all $w|v$.
  This means that $\lambda(v)$ is a multiple of $\lcm_{w|v} \frac{dm}{\gcd(dm,d_w)} = \frac{dm}{\gcd_{w\mid v}\gcd(dm,d_w)}$.
  Since $\gcd_{w\mid v} d_w$ divides~$\sum_{w\mid v} d_w = d$, we have $\gcd_{w\mid v}\gcd(dm,d_w) = \gcd_{w\mid v} d_w$.
  To conclude, it remains to show that $\gcd_{w \mid v} d_w = \cycgcd(\Frob(v))$.
  
  Since the prime $v$ is non-exceptional, it is unramified in $F$.
  Pick a representative $g \in G$ of the conjugacy class $\Frob(v)$.
  Orbits of the action of $g$ on $\{1,\dots,d\}$ correspond bijectively to primes $w|v$ of~$F$, and the size of an orbit is the corresponding local degree $d_w$.
  Hence, $\gcd_{w|v} d_w$ is the greatest common divisor of the sizes of the orbits of $g$.
  By definition, this is $\cycgcd(\Frob(v))$.
\end{proof}

\subsection{Setting up the Dirichlet series}
\label{ssn:key-dirichlet}

In this subsection, we set up a Dirichlet series for the counting problem and rewrite it as a sum of Euler products.

We fix a divisor $\tau$ of $M$, which is used in \Cref{ssn:restriction-divalg} to sieve out extensions which are not division algebras.
(For proving just part \ref{thm:main-inner-i} of \Cref{thm:main-inner}, one can take $\tau=1$.)

Let $S$ be a finite set of places of $Z$ containing $\Pex$.
Let $\xi$ be a map $S \to \Z/M\Z$ satisfying conditions \ref{def:Lambda:complex}--\ref{def:Lambda:div} of \Cref{def:Lambda} for all $v\in S$, and such that $\tau$ divides $\xi(v)$ for all $v\in S$.
Define
\[
\sigma_\xi
\coloneqq
\sum_{v\in S}\xi(v)
\]
and let $d(\xi)$ denote the contribution of places in $S$ to the discriminant:
\[
  d(\xi)
  \coloneqq
  \prod_{p \in S}
    \Nm{p}^{
      M
      \Bigl(
        M -
        \gcd\bigl(
          M, \,
          \xi(p)
        \bigr)
      \Bigr)
    }.
\]
Let $\Lambda_{S,\xi,\tau}\subseteq\Lambda$ be the set of maps $\lambda:\Pm\rightarrow\Z/M\Z$ in $\Lambda$ whose restriction to $S$ is $\xi$ and such that $\tau$ divides $\lambda(v)$ for all places $v \in \Pm$.
Via the bijection of \Cref{thm:ext-vs-map}, elements of $\Lambda_{S,\xi,\tau}$ correspond to isomorphism classes of central simple $Z$-algebras of dimension $M^2$ in which $K$ embeds, whose local invariants at places $v \in S$ are given by $\frac {\xi(v)} M$, and whose index divides~$\frac M \tau$.
Moreover, if $\lambda \in \Lambda_{S,\xi,\tau}$ and~$L/K$ is the corresponding extension, then $d(L/Z) = d(\lambda)$.
We are thus led to count maps $\lambda\in\Lambda_{S,\xi,\tau}$ with $d(\lambda) \leq X$.
For this, we introduce the Dirichlet series:
\[
  f_{S,\xi,\tau}(s)
  =
  \sum_{\lambda\in\Lambda_{S,\xi,\tau}}
    d(\lambda)^{-s}.
\]
To specify a map $\lambda\in\Lambda_{S,\xi,\tau}$, we only need to specify its restriction to $\Pm\setminus S$.
Unraveling definitions and using \Cref{lem:csa-inclusion-unramified}, this lets us write:
\begin{equation}
  \label{eqn:f-i}
  f_{S,\xi,\tau}(s)
  =
  d(\xi)^{-s}
  \sum_{
    \text{(see below)}
  }
    \,\,
    \prod_{\substack{p\in\Pm\setminus S\\\lambda(p)\neq0}}
      \Nm{p}^{
        -s M
        \Bigl(
          M - \gcd \bigl( M, \, \lambda(p) \bigr)
        \Bigr)
      }.
\end{equation}
where the sum is taken over finitely supported maps $\lambda : \Pm \setminus S \to \Z/M\Z$ such that:
\begin{enumerate}[label=(\roman*)]
  \item
    \label{cond-lambda-i}
    $\sigma_\xi + \sum_{p\in\Pm\setminus S}\lambda(p) = 0$ in $\Z/M\Z$ 
  \item
    \label{cond-lambda-ii}
    for all primes $p \in \Pm\setminus S$, both $\tau$ and $\frac {d m} {\cycgcd(\Frob(p))}$ divide $\lambda(p)$
\end{enumerate}
We encode condition \ref{cond-lambda-i} by the character sum
\[
  \frac1M
  \sum_{\charindex=0}^{M-1}
    \e\leftl(
      \frac\charindex M
      \left(
        \sigma_\xi +
        \sum_{p\in\Pm\setminus S}
          \lambda(p)
      \right)
    \rightr)
  =
  \frac1M
  \sum_{\charindex=0}^{M-1}
    \e\leftl(
      \frac{\charindex\sigma_\xi}M
    \rightr)
    \prod_{\substack{
      p\in\Pm\setminus S\\
      \lambda(p)\neq0
    }} 
      \e\leftl(
        \frac{\charindex\lambda(p)}M
      \rightr).
\]
which equals $1$ if \ref{cond-lambda-i} holds and $0$ otherwise.
We also let
$
  \eta_{\tau, p}
  \coloneqq
  \lcm(
    \frac{d m}{\cycgcd(\Frob(p))}, \,
    \tau
  )
$, so that condition~\ref{cond-lambda-ii} rewrites as $\eta_{\tau, p} | \lambda(p)$.
Plugging this into \Cref{eqn:f-i} lets us write the Dirichlet series as a finite sum of Euler products:
\begin{equation}
  \label{eqn:f-sum-eulprod}
  f_{S,\xi,\tau}(s)
  =
  \frac{d(\xi)^{-s}}M
  \sum_{\charindex=0}^{M-1}
    \left(
      \e\leftl(
        \frac{\charindex\sigma_\xi}M
      \rightr)
      \prod_{p\in\Pm\setminus S}
        \,\,
        \sum_{\substack{
          \lambda\in\Z/M\Z\\
          \eta_{\tau, p} \mid \lambda
        }}
          \,\,
          \e\leftl(
            \frac{\charindex\lambda}M
          \rightr)
          \Nm{p}^{-sM\big(M-\gcd(M,\lambda)\big)}
    \right)
  .
\end{equation}
We give a name to the Euler factor:
\begin{equation*}
  \mathbf{f}_{\tau,\charindex,p}(s)
  \coloneqq
  \sum_{\substack{
    \lambda \in \Z/M\Z \\
    \eta_{\tau, p} \mid \lambda
  }}
    \e\leftl(
      \frac{\charindex\lambda}M
    \rightr)
    \Nm{p}^{
      -sM
      \big(
        M-\gcd(M,\lambda)
      \big)
    }.
\end{equation*}
\Cref{eqn:f-sum-eulprod} then rewrites as:
\begin{equation}
  \label{eqn:split-f-into-chars}
  f_{S,\xi,\tau}(s)
  =
  \frac{d(\xi)^{-s}}M
  \sum_{\charindex=0}^{M-1}
    \e\leftl(
      \frac{\charindex\sigma_\xi}M
    \rightr)
    \prod_{p\in\Pm\setminus S}
      \mathbf{f}_{\tau,\charindex,p}(s).
\end{equation}
Split up the sum defining $\mathbf{f}_{\tau,\charindex,p}(s)$, grouping values of $\lambda$ according to the greatest common divisor $g = \gcd(M,\lambda)$, which must satisfy $\eta_{\tau, p} \mid g$ and $g \mid M$, and writing $\lambda = g \lambda'$ with
$
  \lambda'
  \in
  \left(
    \Z/\frac Mg\Z
  \right)^{\times}
$.
We have:
\begin{align*}
  \mathbf{f}_{\tau,\charindex,p}(s)
  &=
  \sum_{\substack{
    g\geq1 \textnormal{  such that} \\
    \eta_{\tau, p} \mid g \textnormal{ and } g \mid M
  }}
    \,\,
    \sum_{\lambda' \in \left(\Z/\frac Mg\Z\right)^\times}
      \e\leftl(
        \frac{\charindex g\lambda'}M
      \rightr)
      \Nm{p}^{-sM(M-g)}
  \\
  &=
  1
  +
  \sum_{\substack{
    1 \leq g < M \textnormal{ such that} \\
    \eta_{\tau, p} \mid g \textnormal{ and } g \mid M
  }}
  \,\,
    \sum_{\lambda' \in \left(\Z/\frac Mg\Z\right)^\times}
      \e\leftl(
        \frac{\charindex g\lambda'}M
      \rightr)
      \Nm{p}^{-sM(M-g)}
  .
\end{align*}
Finally, for $g|M$, define
\[
  g_{\charindex}\leftl(
    \frac M g
  \rightr)
  =
  \sum_{
    \lambda'
    \in
    (\Z/\frac M g\Z)^{\times}
  }
    \e\leftl(
      \frac
        {\charindex g \lambda'}
        {M}
    \rightr)
\]
so that
\begin{equation}
  \label{eqn:eulfact-f}
  \mathbf{f}_{\tau,\charindex,p}(s)
  =
    1
    +
    \sum_{\substack{
      1 \leq g < M \textnormal{ such that} \\
      \eta_{\tau, p} \mid g \textnormal{ and } g \mid M
    }}
      \,
      g_{\charindex}\leftl(
        \frac M g
      \rightr)
      \Nm{p}^{-s M (M - g)}
  .
\end{equation}

\begin{remark}
  \label{rmk:general-gchi}
  The multiplicative function $g_{\charindex}$ takes the following value at an arbitrary $\frac M g$ dividing~$M$:
  \[
    g_{\charindex}\leftl(
      \frac M g
    \rightr)
    =
    \frac
      {\varphi\leftl(
        \frac M g
      \rightr)}
      {
        \varphi\leftl(
          \frac M {\gcd(g \charindex, \, M)}
        \rightr)
      }
    \cdot
    \mu\leftl(
      \frac M {\gcd(g \charindex, \, M)}
    \rightr).
  \]
  For instance, the function $g_0$ is Euler's totient function $\phi$, and $g_1$ is the Möbius function $\mu$.

  %
\end{remark}

\subsection{Analytic properties of the Dirichlet series}
\label{ssn:analytic-props}

All notations are as in the previous subsection.
We define $a = M \left( M - \Mu \right)$ as in \Cref{thm:main-inner}.
Our goal is to describe the behavior of the Euler products $\prod_{p\in\Pm\setminus S} \, \mathbf{f}_{\tau,\charindex,p}(s)$ in the half-plane $\{\Re(s) \geq \frac 1 a\}$.
We first show that the most significant term in each Euler factor $\mathbf f_{\tau,\charindex,p}(s)$ is determined by the Frobenius automorphism associated to $p$ (\Cref{lem:fapprox}).
We later use this information to relate the analytic properties of $\prod_p \mathbf f_{\tau,\charindex,p}(s)$ to those of a product of powers of the Artin L-functions associated to the field extension $\hat F/Z$ (\Cref{lem:analytic-stuff} and \Cref{lem:approx}).
We then obtain asymptotics for the distribution of inner Galois extensions of $L/K$ with degree $n$ and center~$Z$ using Delange's Tauberian theorem (\Cref{cor:asymptotics}).
This establishes most of \iref{thm:main-inner}{thm:main-inner-i}, the only missing point being that the leading coefficient is positive when such an extension exists.

Let $\psi_{\tau,\charindex} \colon G \to \C$ be the following class function:
\begin{equation}
  \label{eqn:def-psi}
  \psi_{\tau,\charindex}(g)
  \coloneqq
  \begin{cases}
    g_\charindex(u)
    &
    \text{if } \lcm\leftl(\frac {d m} {\cycgcd(g)}, \, \tau\rightr) \text{ divides } \Mu, \\
    0 &\text{otherwise}.
  \end{cases}
\end{equation}
We use $\psi_{\tau,\charindex}$ to approximate the Euler factor $\mathbf f_{\tau,\charindex,p}(s)$ as follows:

\begin{lemma}
  \label{lem:fapprox}
  There is a constant $\epsilon>0$ such that, for $\Re(s) \geq \frac 1 a$:
  \[
    \mathbf f_{\tau,\charindex,p}(s)
    =
    1
    +
    \psi_{\tau,\charindex}\big(
      \Frob(p)
    \big)
    \Nm p^{-as}
    +
    \O\leftl(
      \Nm p^{-(1+\varepsilon)as}
    \rightr),
  \]
  where both $\epsilon$ and the implied constant in the $\O$-term are independent of $p$ and $s$.
\end{lemma}
\begin{proof}
  Consider the expression of $\mathbf{f}_{\tau,\charindex,p}(s)$ given in \Cref{eqn:eulfact-f}.
  If a proper divisor $g$ of $M$ occurs for a summand in $\mathbf{f}_{\tau,\charindex,p}(s)$, then $\frac{M}{g}$ divides $\frac{M}{\eta_{\tau,p}}$, which divides $\frac{M}{dm/\cycgcd(\Frob(p))} = j \cdot \cycgcd(\Frob(p))$, which divides~$j U$.
  By definition, the smallest divisor of $j U$ besides $1$ is $u$.
  Hence, the largest occurring proper divisor $g$ of~$M$ is at most $\Mu$.
  The corresponding summand, if it occurs, is $g_\charindex(u)\Nm{p}^{-as}$.
  It follows that, for some $\varepsilon > 0$:
  \begin{align*}
    \mathbf{f}_{\tau,\charindex,p}(s)
    &
    =
    \begin{cases}
      1
      +
      g_\charindex(u) \Nm{p}^{-as}
      +
      \O\leftl(
        \Nm{p}^{-(1+\varepsilon)as}
      \rightr)
      &
      \text{if }
      \eta_{\tau, p} \mid \Mu
      \\
      1
      +
      \O\leftl(
        \Nm{p}^{-(1+\varepsilon)as}
      \rightr)
      &
      \text{otherwise}
    \end{cases}
    \\
    &
    =
    1
    +
    \psi_{\tau,\charindex}\big(
      \Frob(p)
    \big)
    \Nm{p}^{-as}
    +
    \O\leftl(
      \Nm{p}^{-(1+\varepsilon)as}
    \rightr)
    .
    \qedhere
  \end{align*}
\end{proof}

\paragraph{An analytic lemma.}

We now prove \Cref{lem:analytic-stuff}, in which we approximate ``Frobenian'' Euler products using products of Artin L-functions.
(See \cite[Section~2]{prescribed-norms} for an introduction to Frobenian functions.)
This is used later to analyze the behavior of $\prod_p f_{\tau,\charindex,p}(s)$.

\begin{definition}
  When $z$ is a complex number, we define the holomorphic non-vanishing function $s \mapsto (s-1)^z$ on the open half-plane $\{\Re(s) > 1\}$ as $s \mapsto \exp(z\log(s-1))$, where $\log$ is the unique determination of the complex logarithm on the open half-plane $\{\Re(s) > 0\}$ taking real values on the positive real half-line.
\end{definition}

Consider any irreducible representation $\rho$ of $G$ and let $\chi:G\to\C$ be the corresponding character.
It is well-known that the Artin L-function $L(\rho,s)$ is holomorphic non-vanishing for $\Re(s)\geq1$, except for a simple pole at $s=1$ when $\rho$ is the trivial representation. (See \cite[p.\ 225]{cassfr}.)
For every place $p \in \Pm$, let $\mathbf{h}_{\chi,p}(s)$ be the Euler factor at $p$ in the Euler product defining the L-function $L(\rho,s)$ (cf. \cite[Chap.~VII, (10.1)]{neukirch}), so that $L(\rho, s) = \prod_{p\in\Pm}\mathbf{h}_{\chi,p}(s)$.
By definition, the Euler factors $\mathbf{h}_{\chi,p}(s)$ are holomorphic and non-vanishing for $\Re(s)>0$, and the product $\prod_{p\in\Pm}\mathbf{h}_{\chi,p}(s)$ is absolutely convergent when $\Re(s)>1$.
When $p \in \Pm$ is an unramified prime, the Euler factor is given by
\[
  \mathbf{h}_{\chi, p}(s)
  =
  \det\Bigl(
    I
    -
    \rho\big(
      \Frob(p)
    \big)
    \Nm{p}^{-s}
  \Bigr)^{-1}.
\]
Expanding the characteristic polynomial, we obtain the following estimate for $\Re(s)\geq\frac12$:
  \begin{align*}
    \mathbf{h}_{\chi, p}(s)
    &
    =
    \Big(
      1
      -
      \tr\big(
        \rho(\Frob(p))
      \big)
      \Nm{p}^{-s}
      +
      \O\big(
        \Nm{p}^{-2s}
      \big)
    \Big)^{-1}
    \\
    &
    =
    1
    +
    \tr\big(
      \rho(\Frob(p))
    \big)
    \Nm{p}^{-s}
    +
    \O\big(
      \Nm{p}^{-2s}
    \big)
    \\
    &
    =
    1
    +
    \chi(\Frob(p))
    \Nm{p}^{-s}
    +
    \O\big(
      \Nm{p}^{-2s}
    \big).
  \end{align*}
  We now consider an arbitrary class function $\psi:G\to\C$.
  We define its \emph{average} as its inner product with the trivial character:
\[
  \avg(\psi)
  \coloneqq
  \frac{1}{\card{G}}
  \sum_{g \in G}
    \psi(g).
\]
  Recall that $\psi$ decomposes as a sum over the finitely many irreducible characters $\chi$ of $G$:
  \[
    \psi
    =
    \sum_{\chi}
      \psichi
      \chi.
  \]
  We extend the definition of $\mathbf{h}_{\psi, p}$ to class functions $\psi$ which are not irreducible characters, by setting
  \[
    \mathbf{h}_{\psi, p}
    \coloneqq
    \prod_{\chi}
      \mathbf{h}_{\chi, p}^\psichi
  \]
  in which the power is interpreted as follows: for every irreducible character $\chi$, the function $\mathbf{h}_{\chi, p}$ is holomorphic non-vanishing on the open simply connected subset $\{\Re(s) > 0\}$, and thus admits a logarithm $\log \mathbf{h}_{\chi, p}$; note that $\mathbf{h}_{\chi, p}(s) \underset{s \to \infty}\longrightarrow 1$ and choose the logarithm specifically so that  $\log \mathbf{h}_{\chi, p}(s) \underset{s \to \infty}\longrightarrow 0$, which uniquely determines it; now set $\mathbf{h}_{\chi, p}^\psichi = \exp(\psichi \log \mathbf{h}_{\chi, p})$.

\begin{lemma}
  \label{lem:analytic-stuff}
  Let $\psi$ be a class function $G \to \C$.
  Then:
  \begin{enumerate}[label=(\roman*)]
    \item
      \label{lem:analytic-stuff-i}
      $\mathbf{h}_{\psi, p}$ is holomorphic non-vanishing on the open half-plane $\{\Re(s) > 0\}$ for all places $p\in\Pm$.
    \item
      \label{lem:analytic-stuff-ii}
      For $\Re(s) \geq \frac12$, and all unramified primes $p\in\Pm$,
      \[
        \mathbf{h}_{\psi, p}(s)
        =
        1
        +
        \psi\bigl(\Frob(p)\bigr)
        \Nm{p}^{-s}
        +
        \O_\psi\leftl(
          \Nm{p}^{-2s}
        \rightr)
      \]
      where the implied constant in the $\O$-term is independent of both $p$ and $s$.
    \item
      \label{lem:analytic-stuff-iii}
      The product
      $
        (s-1)^{\avg(\psi)}
          \prod_{\substack{
            p \in \Pm
          }}
            \mathbf{h}_{\psi, p}(s)
      $,
      which is absolutely convergent for $\Re(s) > 1$, extends to a holomorphic non-vanishing function $h_{\psi}$ on the closed half-plane $\{ \Re(s) \geq 1 \}$ with
      \[
        h_\psi(1) = (\Res_{s=1}\zeta_Z(s))^{\avg(\psi)}\prod_{\chi\neq1} L(\chi,1)^{\langle\psi,\chi\rangle}.
      \]
  \end{enumerate}
\end{lemma}

\begin{proof}
  We have shown these properties for irreducible characters $\psi=\chi$ above.
  Now, consider an arbitrary class function $\psi = \sum_{\chi} \psichi \chi$.
  Point \ref{lem:analytic-stuff-i} follows immediately from the definition.
  For \ref{lem:analytic-stuff-ii}, we compute
  \begin{align*}
    \mathbf{h}_{\psi, p}(s)
    &
    =
    \prod_{\chi}
      \left(
        1
        +
        \chi(\Frob(p))
        \Nm{p}^{-s}
        +
        \O\big(
          \Nm{p}^{-2s}
        \big)
      \right)^\psichi
    \\
    &
    =
    1
    +
    \sum_{\chi}
      \psichi
      \chi(\Frob(p))
      \Nm{p}^{-s}
    +
    \O\big(
      \Nm{p}^{-2s}
    \big)
    \\
    &
    =
    1
    +
    \psi(\Frob(p))
    \Nm{p}^{-s}
    +
    \O\big(
      \Nm{p}^{-2s}
    \big).
  \end{align*}
  For \ref{lem:analytic-stuff-iii}, note that
  \begin{align*}
    (s-1)^{\avg(\psi)}
      \displaystyle
      \prod_{\substack{
        p \in \Pm
      }}
        \mathbf{h}_{\psi, p}(s)
    &
    =
    (s-1)^{\langle \psi, 1\rangle}
      \displaystyle
      \prod_{\substack{
        p \in \Pm
      }}
        \,\,
        \prod_{\chi}
          \mathbf{h}_{\chi, p}^\psichi(s)
    =
    h_1^{\langle \psi, 1\rangle}
    \cdot
    \prod_{\chi \neq 1}
      h_{\chi}^\psichi.
  \end{align*}
  Each of the finitely many factors extends to a holomorphic non-vanishing function on the closed half-plane $\{ \Re(s) \geq 1 \}$ as shown above.
  This establishes \ref{lem:analytic-stuff-iii} with $h_{\psi} = \prod_{\chi} h_{\chi}^\psichi$.
\end{proof}

\begin{remark}
  A very similar result can be found in \cite[Proposition~2.3]{prescribed-norms}, with essentially the same proof.
  The main difference is that they use Euler factors of the form $1+\psi(\Frob(p))\Nm{p}^{-s}$, which (as they point out in their Remark~2.4) can vanish for small primes.
  They therefore need to exclude all primes $p$ with $\Nm{p}\leq\max_{g\in G}|\psi(g)|$, which would be somewhat inconvenient for us. Moreover, our expression for $h_\psi(1)$ is a finite product, whereas the expression in \cite[Equation~(2.5)]{prescribed-norms} is an infinite product that is only conditionally convergent.
\end{remark}

\paragraph{Application.}

Denote by $h_{\psi_{\tau, \charindex}}$ the function associated as in \iref{lem:analytic-stuff}{lem:analytic-stuff-iii} to the class function~$\psi_{\tau, \charindex}$ from \Cref{eqn:def-psi}.

\begin{lemma}
  \label{lem:approx}
  The function
  \begin{equation}
    \label{eqn:def-tildef}
    \tilde f_{S,\tau,\charindex}(s)
    \coloneqq
    (s-1)^{\avg(\psi_{\tau,\charindex})}
    \prod_{p\in\Pm\setminus S}
      \mathbf{f}_{\tau,\charindex,p}
      \left(
        \frac s a
      \right)
  \end{equation}
  extends to a non-vanishing holomorphic function on the closed half-plane $\{\Re(s) \geq 1\}$, given by the expression:
  \begin{equation}
    \label{eqn:formula-tildef}
    \tilde f_{S, \tau, \charindex}(s)
    =
    h_{\psi_{\tau, \charindex}}(s)
    \left(
      \prod_{\substack{
        p \in S
      }}
        \mathbf{h}_{\psi_{\tau, \charindex}, p}(s)^{-1}
    \right)
    \left(
      \prod_{
        p \in \Pm \setminus S
      }
        \mathbf{f}_{\tau,\charindex,p}\leftl(
          \frac s a
        \rightr)
        \mathbf{h}_{\psi_{\tau, \charindex}, p}(s)^{-1}
    \right).
  \end{equation}
  in which the infinite product is absolutely convergent on the closed half-plane $\{\Re(s) \geq 1\}$.
\end{lemma}

\begin{proof}
  For $\Re(s) > 1$, the expression for $\tilde f_{S, \tau, \charindex}(s)$ in \Cref{eqn:formula-tildef} follows directly from unfolding the definition of $h_{\psi_{\tau, \charindex}}$ (\iref{lem:analytic-stuff}{lem:analytic-stuff-iii}).
  The first factor $h_{\psi_{\tau, \charindex}}$ is holomorphic non-vanishing on the closed half-plane $\{\Re(s) \geq 1\}$ by \iref{lem:analytic-stuff}{lem:analytic-stuff-iii}.
  The second factor (the product over primes in $S$) is holomorphic non-vanishing on the open half-plane $\{\Re(s) > 0\}$ as a finite product of such functions, cf. \iref{lem:analytic-stuff}{lem:analytic-stuff-i}.
  We now check that the third factor, which is an infinite product, is absolutely convergent on the closed half-plane $\{\Re(s) \geq 1\}$.
  To this end, we describe the asymptotic behavior of its factors as $\Nm{p} \to \infty$.
  By \Cref{lem:fapprox}, we have:
  \[
    \mathbf f_{\tau,\charindex,p}(s)
    = 1 + \psi_{\tau,\charindex}(\Frob(p)) \Nm p^{-as} + O(\Nm p^{-(1+\varepsilon)as}).
  \]
  Any prime not in $S$ is unramified, thus by \iref{lem:analytic-stuff}{lem:analytic-stuff-ii}, we have:
  \[
    \mathbf{h}_{\psi_{\tau, \charindex}, p}(s)
    =
    1
    +
    \psi_{\tau, \charindex}\big(\Frob(p)\big)
    \Nm{p}^{-s}
    +
    \O\leftl(
      \Nm{p}^{-2s}
    \rightr).
  \]
  Therefore:
  \begin{align*}
    \mathbf{f}_{\tau,\charindex,p}\leftl(
      \frac s a
    \rightr)
    \mathbf{h}_{\psi_{\tau, \charindex}, p}(s)^{-1}
    &
    =
    1
    +
    \O\big(
      \Nm{p}^{-\min(2, \, 1+\varepsilon) s}
    \big).
  \end{align*}
  Hence, the infinite product in \Cref{eqn:formula-tildef} is absolutely convergent on the open half-plane $\{ \Re(s) > \frac 1 {\min(2, \, 1+\varepsilon)}\}$.
\end{proof}

\begin{remark}
  \Cref{lem:approx} can be interpreted as saying that
  $
  \prod_{p\in\Pm\setminus S}
    \mathbf{f}_{\tau,\charindex,p}(s)
  $
  has its rightmost ``pole'' at $s=a$, and that the ``order'' of this ``pole'' is the rational number $\avg(\psi_{\tau,\charindex})$.
  However, $a$ is often not an actual pole as the infinite product is not meromorphic on the closed half-plane when $\avg(\psi_{\tau, \charindex})$ is not an integer.
\end{remark}

Since $u$ is a prime number, the number $g_\charindex(u)$ is easy to compute:
\begin{equation}
  \label{eqn:gchiu}
  g_{\charindex}(u)
  =
  \sum_{\lambda'\in(\Z/u\Z)^\times}
    \e\leftl(
      \frac{\charindex\lambda'}{u}
    \rightr)
  =
  \left(
  \sum_{\lambda'\in\Z/u\Z}
    \e\leftl(
      \frac{\charindex\lambda'}{u}
    \rightr)
  \right)
  -
  1
  =
  \left\lbrace
    \begin{array}{cl}
      u-1 & \text{if } u|\charindex \\
      -1 & \text{otherwise.}
    \end{array}
  \right.
\end{equation}

\begin{lemma}
  \label{lem:avg-psi-computation}
  Let $k \in \{0, \ldots, M-1\}$.
  For $\tau=1$, we have:
  \[
    \avg(\psi_{1,\charindex})
    =
    \begin{cases}
      (u-1)\beta & \text{if } u|\charindex,\\
      -\beta & \text{otherwise}
    \end{cases}
  \]
  and, for any $\tau \mid M$, we have
  $
    \avg(\psi_{\tau,\charindex}) \leq (u-1)\beta
  $.
\end{lemma}

\begin{proof}
  By \Cref{eqn:def-psi}, we have
  \begin{align*}
    \avg(\psi_{\tau,\charindex})
    =
    \frac1{\card G}
    \sum_{g \in G}
      \psi_{\tau,\charindex}(g)
    =
    \frac{g_\charindex(u)}{\card G}
    \card{\left\{
      g\in G
      \verti
      \lcm\left(
        \frac{dm}{\cycgcd(g)},\tau
      \right)
      \textnormal{ divides }
      \frac{M}{u}
    \right\}}.
  \end{align*}
  The claims follow using \Cref{eqn:gchiu} and \Cref{eqn:beta-otherdef}.
\end{proof}

\paragraph{Application of Delange's Tauberian theorem.}
We have the following expression for the Dirichlet series $f_{S,\xi,\tau}$ counting elements of $\Lambda_{S,\xi,\tau}$ by discriminant:
\begin{align*}
  f_{S,\xi,\tau}(s)
  &
  =
  \frac{d(\xi)^{-s}}M
  \sum_{\charindex=0}^{M-1}
    \e\leftl(
      \frac{\charindex\sigma_\xi}M
    \rightr)
    \prod_{p\in\Pm\setminus S}
      \mathbf{f}_{\tau,\charindex,p}(s)
  &
  \text{
    by
    \Cref{eqn:split-f-into-chars}\,%
  }
  \\
  &
  =
  \frac{d(\xi)^{-s}}M
  \sum_{\charindex=0}^{M-1}
    \e\leftl(
      \frac{\charindex\sigma_\xi}M
    \rightr)
    \tilde f_{S,\tau,\charindex}(as)
    (as-1)^{-\avg(\psi_{\tau,\charindex})}
  &
  \text{
    by
    \Cref{eqn:def-tildef}.%
  }
\end{align*}

This already implies a weak form of \iref{thm:main-inner}{thm:main-inner-i}:

\begin{corollary}
  \label{cor:asymptotics}
  We have the asymptotic estimate:
  \[
    \card{
      \left\{
        \lambda \in \Lambda_{S,\xi,\tau}
        \verti
        d(\lambda) \leq X
      \right\}
    }
    =
    C_{S,\xi,\tau}
    X^{1/a}
    \log(X)^{(u-1)\beta-1}
    +
    o\leftl(
      X^{1/a}
      \log(X)^{(u-1)\beta-1}
    \rightr)
  \]
  where
  \begin{equation}
    \label{eqn:expr-constant}
    C_{S,\xi,\tau}
    =
    \frac
    {1}
    {
      a^{(u-1)\beta-1}
      \cdot
      \Gamma\big(
        (u-1)\beta
      \big)
    }
    \cdot
    \frac 
    {
      d(\xi)^{-1/a}
    }
    M
    \sum_{\substack{
      0 \leq \charindex \leq M-1 \\
      \avg(\psi_{\tau, \charindex}) = (u-1)\beta
    }}
      \e\leftl(
        \frac{\charindex\sigma_\xi}M
      \rightr)
      \tilde f_{S,\tau,\charindex}(1)
    .
  \end{equation}
\end{corollary}

\begin{proof}
  We apply Delange's Tauberian theorem \cite[Théorème~III]{delange} as follows:
  Let
  \[
    \alpha(t)
    \coloneqq
    |\{\lambda \in \Lambda_{S,\xi,\tau} \ |\  d(\lambda) \leq e^t\}|,
  \]
  so that the function $f(s)$ in Delange's notation is
  \begin{align*}
    f(s)
    & =
    \int_0^\infty e^{-st}\alpha(t)\mathrm d t
    \ =\ s^{-1}\sum_{\lambda\in\Lambda_{S,\xi,\tau}}d(\lambda)^{-s}
    \ =\ s^{-1} f_{S,\xi,\tau}(s) \\
    & =
    s^{-1}
    \frac{d(\xi)^{-s}}M
    \sum_{\charindex=0}^{M-1}
      a^{-\avg(\psi_{\tau,\charindex})}
      \e\leftl(
        \frac{\charindex\sigma_\xi}M
      \rightr)
      \tilde f_{S,\tau,\charindex}(as)
      \left(s-\frac1a\right)^{-\avg(\psi_{\tau,\charindex})}.
  \end{align*}
  By \Cref{lem:avg-psi-computation}, the largest value taken by $\avg(\psi_{\tau, k})$ is $(u-1) \beta$.
  Hence, in Delange's notation, $\omega=(u-1)\beta$, and the function $g(s)$ in front of the factor $(s-\frac1a)^{-(u-1) \beta}$ is obtained by summing over values of~$k$ at which this maximal average is reached:
  \[
    g(s) =
    s^{-1} a^{-(u-1)\beta}
    \frac{d(\xi)^{-s}}M
    \sum_{\substack{
      0 \leq \charindex \leq M-1 \\
      \avg(\psi_{\tau, \charindex}) = (u-1)\beta
    }}
      \e\leftl(
        \frac{\charindex\sigma_\xi}M
      \rightr)
      \tilde f_{S,\tau,\charindex}(as).
  \]
  Then, Delange's theorem implies:
  \begin{align*}
  &|\{\lambda \in \Lambda_{S,\xi,\tau}\ |\ d(\lambda) \leq X\}|
  \\
  =\ &
  \alpha(\log X) \\
  =\ &
  \left(\frac{g(\frac1a)}{\Gamma((u-1)\beta)} + o(1)\right)
    X^{1/a}
    \log(X)^{(u-1) \beta  - 1}
  \\
  =\ &
    \left(
      \frac1{a^{(u-1)\beta-1} \Gamma\big((u-1)\beta\big)}
      \cdot
      \frac{d(\xi)^{-1/a}}M
      \sum_{\substack{
        0 \leq \charindex \leq M-1 \\
        \avg(\psi_{\tau, \charindex}) = (u-1)\beta
      }}
        \e\leftl(
          \frac{\charindex\sigma_\xi}M
        \rightr)
        \tilde f_{S,\tau,\charindex}(1)
      + o(1)
    \right)
    X^{1/a}
    \log(X)^{(u-1) \beta  - 1}.
  \end{align*}
  (Technically, Delange only states the asymptotic equivalence $\alpha(\log X)\sim C X^{1/a}\log(X)^{(u-1)\beta-1}$, assuming that the resulting constant $C$ is nonzero.
  However, one can check that the claim $\alpha(\log X)=(C+o(1))X^{1/a}\log(X)^{(u-1)\beta-1}$ follows in the same way, even when $C=0$.
  Alternatively, one can apply Delange's theorem to two functions $f_1(s)$ and $f_2(s)$ with $f(s)=f_1(s)-f_2(s)$ and such that~$f_2(s)$ has a positive constant $C_2$, and then subtract the resulting asymptotic statements.
  There are many sources of such functions $f_1(s)$, $f_2(s)$: for example one can take $f_2(s) \coloneq \prod_{p\in\Pm\setminus S}\mathbf{f}_{1,0,p}(s)$.)
\end{proof}

The real number $C_{S,\xi,\tau}$ is nonnegative because of its combinatorial interpretation.
However, at this point, we have not established that 
$C_{S,\xi,\tau}$ is nonzero: proving this fact (under the assumption that an extension indeed exists) is the focus of \Cref{ssn:positivity-leading-coeff}, and is the only piece of \iref{thm:main-inner}{thm:main-inner-i} that is still missing.

Note that \Cref{eqn:formula-tildef} gives an expression for $\tilde f_{S,\tau,\charindex}(1)$ which can be used to compute $C_{S, \xi, \tau}$:
\begin{equation}
  \label{eqn:formula-tildef-at1}
  \tilde f_{S, \tau, \charindex}(1)
  =
  h_{\psi_{\tau, \charindex}}(1)
  \left(
    \prod_{\substack{
      p \in S
    }}
      \mathbf{h}_{\psi_{\tau, \charindex}, p}(1)^{-1}
  \right)
  \left(
    \prod_{
      p \in \Pm \setminus S
    }
      \mathbf{f}_{\tau,\charindex,p}\leftl(
        \frac 1 a
      \rightr)
      \cdot
      \mathbf{h}_{\psi_{\tau, \charindex}, p}(1)^{-1}
  \right).
\end{equation}
The value of $h_{\psi_{\tau, \charindex}}(1)$ can itself be computed from the values (resp. residue for the trivial character) at $s=1$ of the Artin L-functions associated to the irreducible characters of $G$, cf.\ \iref{lem:analytic-stuff}{lem:analytic-stuff-iii}.

\begin{remark}
  \label{rk:constant-computation-when-psi-constant}
  An interesting special case arises when $\psi_{\tau,\charindex}  = \avg(\psi_{\tau,\charindex}) \in \C$ is a constant class function (i.e., a multiple of the trivial character) --- this is in particular always the case if $F=Z$.
  In that case, we can use the expression for the residue $\Res_{s=1}\zeta_Z(s)$ of the Dedekind zeta function of $Z$ at~$s=1$ given by the class number formula (for instance, it is $1$ if $Z=\Q$) to get a more concrete expression for~$\tilde f_{S, \tau, \charindex}(1)$ (and hence for the leading coefficient $C_{S,\xi,\tau}$):
  \[
    \tilde f_{S, \tau, \charindex}(1)
    =
    \left(
      \big(
        \Res_{s=1}\zeta_Z(s)
      \big)
      \cdot
      \prod_{\substack{
        p \in S
      }}
        \left(
          1 - \frac 1 {\Nm{p}}
        \right)
    \right)^{\psi_{\tau, \charindex}}
    \left(
      \prod_{
        p \in \Pm \setminus S
      }
        \mathbf{f}_{\tau,\charindex,p}\leftl(
          \frac 1 a
        \rightr)
        \cdot
        \leftl(
          1 - \frac 1 {\Nm{p}}
        \rightr)^{\psi_{\tau, \charindex}}
    \right).
  \]
\end{remark}

\subsection{Positivity of the leading coefficient}
\label{ssn:positivity-leading-coeff}

All notations are as in the previous subsection, and moreover we fix $\tau = 1$.
We also assume that there exists an inner Galois extension $L/K$ of degree $n$ with $Z(L) = Z$ whose local invariants at places~$v \in S$ are given by $\frac {\xi(v)} M$, and we denote by $\lambda_0$ the finitely supported map $\Pm \to \Z/M\Z$ corresponding to this extension.
(We refer to \Cref{thm:criteria-existence} for criteria to check whether such an extension exists.)
Our strategy to prove that the leading coefficient $C_{S, \xi, 1}$ in \Cref{cor:asymptotics} is nonzero relies on the following lemma:

\begin{lemma}
  \label{lem:extend-S-xi}
  If $S'$ is a finite set of places containing $S$, and $\xi' : S' \to \Z/M\Z$ is a map extending~$\xi$, then $C_{S', \xi', 1} \leq C_{S, \xi, 1}$.
  In particular, if $C_{S', \xi', 1}$ is nonzero, then $C_{S, \xi, 1}$ is nonzero.
\end{lemma}

\begin{proof}
  When we extend $S$ and $\xi$, we are putting more constraints on the extensions we are counting and therefore there are fewer of them, i.e., $\Lambda_{S',\xi',1} \subseteq \Lambda_{S,\xi,1}$.
  This implies that:
  \[
    \card{
      \left\{
        \lambda \in \Lambda_{S',\xi',1}
        \verti
        d(\lambda) \leq X
      \right\}
    }
    \leq
    \card{
      \left\{
        \lambda \in \Lambda_{S,\xi,1}
        \verti
        d(\lambda) \leq X
      \right\}
    }
  \]
  and thus $C_{S', \xi', 1} \leq C_{S, \xi, 1}$ by \Cref{cor:asymptotics}.
\end{proof}

Instead of defining a new set $S'$ and a map $\xi' : S' \to \Z/M\Z$, we use the notational shortcut of repeatedly adding places to $S$ and extending $\xi$, until we can ensure that $C_{S, \xi, 1}$ is positive.
Recall from \Cref{eqn:expr-constant} and \Cref{lem:avg-psi-computation} that:
\[
  C_{S,\xi,1}
  =
  \frac
  {1}
  {
    a^{(u-1)\beta-1}
    \cdot
    \Gamma\big(
      (u-1)\beta
    \big)
  }
  \cdot
  \frac 
  {
    d(\xi)^{-1/a}
  }
  M
  \sum_{\substack{
    0 \leq \charindex \leq M-1 \\
    u|\charindex
  }}
    \e\leftl(
      \frac{\charindex\sigma_\xi}M
    \rightr)
    \cdot
    \tilde f_{S,1,\charindex}(1)
  .
\]
Note also that, as soon as $u|\charindex$, the class function $\psi_{1,\charindex}$ defined in \Cref{eqn:def-psi} does not depend on~$\charindex$.
For this reason, we simply let $\psi \coloneqq \psi_{1,0}$.
By \Cref{eqn:formula-tildef}, we have for $u|\charindex$:
\[
  \tilde f_{S, 1, \charindex}(1)
  =
  h_{\psi}(1)
  \left(
    \prod_{\substack{
      p \in S
    }}
      \mathbf{h}_{\psi, p}(1)^{-1}
  \right)
  \left(
    \prod_{
      p \in \Pm \setminus S
    }
      \mathbf{f}_{1,\charindex,p}\leftl(
        \frac 1 a
      \rightr)
      \cdot
      \mathbf{h}_{\psi, p}(1)^{-1}
  \right).
\]

Since the infinite product converges absolutely (\Cref{lem:approx}), we have:
\[
  \prod_{
    \substack{
      p \in \Pm \setminus S \\
      \Nm p > \kappa
    }
  }
    \left(
      \mathbf{f}_{1,\charindex,p}\leftl(
        \frac 1 a
      \rightr)
      \cdot
      \mathbf{h}_{\psi, p}(1)^{-1}
    \right)
  \underset{\kappa \to \infty}{\longrightarrow}
  1.
\]
Fix any $\epsilon \in (0, 1)$, and choose $\kappa$ such that, for all $0\leq\charindex\leq M-1$ with $u \mid \charindex$, we have:
\[
  \left|
    1
    -
    \prod_{
      \substack{
        p \in \Pm \setminus S \\
        \Nm p > \kappa
      }
    }
      \mathbf{f}_{1,\charindex,p}\leftl(
        \frac 1 a
      \rightr)
      \cdot
      \mathbf{h}_{\psi, p}(1)^{-1}
  \right|
  \leq
  \epsilon.
\]
Add to $S$ all the primes $p$ with $\Nm p \leq \kappa$ which are not already in $S$, extending $\xi$ by setting $\xi(p) = \lambda_0(p)$.
Now, define the following nonzero number:
\[
  c
  =
  h_{\psi}(1)
  \left(
    \prod_{\substack{
      p \in S
    }}
      \mathbf{h}_{\psi, p}(1)^{-1}
  \right)
\]
so that, for some complex numbers $\tilde\epsilon_{\charindex}$ of absolute value at most $\epsilon$, we have
$
  \tilde f_{S, 1, \charindex}(1)
  =
  c
  \cdot
  (1+\tilde \epsilon_{\charindex})
$
for all $\charindex$ divisible by $u$.
We also let
\[
  c^+
  =
  \frac
  {
    d(\xi)^{-1/a}
    \cdot
    c
  }
  {
    a^{(u-1)\beta-1}
    \cdot
    \Gamma\big(
      (u-1)\beta
    \big)
  }
\]
so that
\[
  C_{S,\xi,1}
  =
  \frac{c^+}{M}
  \sum_{\substack{
    0 \leq \charindex < M\\
    u|\charindex
  }}
    \e\leftl(
      \frac{\charindex\sigma_\xi}M
    \rightr)
    (1+\tilde\epsilon_\charindex).
\]
Add to $S$ the finitely many primes $p \in P \setminus S$ for which $\lambda_0(p)$ is nonzero, extending $\xi$ by setting $\xi(p) = \lambda_0(p)$.
Since~$\lambda_0$ satisfies condition \ref{def:Lambda:sum} of \Cref{def:Lambda}, we have $\sigma_\xi = \sum_{v\in S}\xi(v) = \sum_{v\in S}\lambda_0(v) = 0$ in $\Z/M\Z$.
Then,
\[
\frac{c^+}{M}
  \sum_{\substack{
    0 \leq \charindex < M \\
    u|\charindex
  }}
    \e\leftl(
      \frac{\charindex\sigma_\xi}M
    \rightr)
= \frac{c^+}{M} \cdot \frac{M}{u} \cdot 1
= \frac{c^+}{u}.
\]
We now have:
\begin{align*}
  \left|
    C_{S,\xi,1}
    -
    \frac {c^+} u
  \right|
  & =
  \left|
    \frac {c^+} M
    \sum_{\substack{
      0 \leq \charindex < M \\
      u|\charindex
    }}
      \e\leftl(
        \frac{\charindex\sigma_\xi}M
      \rightr)
      \tilde\epsilon_\charindex
  \right|
  \leq
  \frac {c^+} M
  \sum_{\substack{
    0 \leq \charindex < M \\
    u|\charindex
  }}
    \leftl|{
        \e\leftl(
          \frac {\charindex \sigma_\xi} M
        \rightr)
        \tilde \epsilon_{\charindex}
    }\rightr|
  \leq
  \frac {c^+} M
  \cdot
  \frac M u
  \cdot
  \epsilon
  =
  \epsilon \cdot \frac {c^+} u
  .
\end{align*}
Since $\epsilon$ was chosen strictly smaller than $1$, it follows that $C_{S,\xi,1}$ is nonzero.
This concludes the proof of \iref{thm:main-inner}{thm:main-inner-i}.


\subsection{Restriction to division algebras}
\label{ssn:restriction-divalg}

All notations are as in \Cref{ssn:key-dirichlet}.
Recall that elements of $\Lambda_{S,\xi,\tau}$ correspond to extensions whose local invariants are all divisible by $\tau$.
Let $\Lambda'_{S,\xi} = \Lambda' \cap \Lambda_{S,\xi,1}$ be the set of maps corresponding to central simple $Z$-algebras of dimension $M^2$ in which $K$ embeds, whose local invariants at places $v \in S$ are given by $\frac {\xi(S)} M$, and which are division algebras.
Note that (cf.\ condition \ref{def:Lambda:gcd} in \Cref{def:Lambda}):
\[
  \Lambda'_{S,\xi}
  =
  \Lambda_{S,\xi,1}
  \setminus
  \bigcup_{\substack{
    \tau|M \\
    \tau>1
  }}
    \Lambda_{S,\xi,\tau}.
\]
By the Möbius inversion formula, we get:
\[
  \card{
    \left\{
      \lambda \in \Lambda'_{S,\xi}
      \verti
      d(\lambda) \leq X
    \right\}
  }
  =
  \sum_{\tau|M}
    \mu(\tau)
    \card{
      \left\{
        \lambda \in \Lambda_{S,\xi,\tau}
        \verti
        d(\lambda) \leq X
      \right\}
    }.
\]
By \Cref{cor:asymptotics}, this implies:
\[
  \card{
    \left\{
      \lambda \in \Lambda'_{S,\xi}
      \verti
      d(\lambda) \leq X
    \right\}
  }
  =
  C'_{S,\xi}
  X^{1/a}
  \log(X)^{(u-1)\beta-1}
  +
  o\leftl(
    X^{1/a}
    \log(X)^{(u-1)\beta-1}
  \rightr)
\]
where
\[
  C'_{S,\xi}
  \coloneqq
  \sum_{\tau|M}
    \mu(\tau)
    C_{S,\xi,\tau}.
\]

All that is left is to show that $C'_{S,\xi}$ is positive when an extension exists.
We assume that $\Lambda'_{S,\xi}$ is not empty, and we fix an element $\lambda_0 \in \Lambda'_{S,\xi}$ (see \Cref{thm:criteria-existence} to see what this hypothesis means in terms of $S$ and $\xi$).
Extend $S$ by adding to it the finitely many primes $p$ of $Z$ not in~$S$ at which~$\lambda_0(p)$ is nonzero, and set $\xi(p) = \lambda_0(p)$ at these primes.
Since $\lambda_0 \in \Lambda'$, we now have $\gcd_{v \in S} \xi(v) = 1$.
Therefore, all central simple algebras associated to maps agreeing with $\xi$ on~$S$ are division algebras.
This ensures that for all divisors $\tau$ of $M$ besides $1$, we have $C_{S,\xi,\tau} = 0$, and thus $C'_{S,\xi} = C_{S,\xi,1}$.
Combined with \Cref{ssn:positivity-leading-coeff}, this implies the positivity of $C'_{S,\xi}$ (and this holds for the original ``non-extended''~$S$ and $\xi$ by a straightforward variant of \Cref{lem:extend-S-xi}).

We have now proved \iref{thm:main-inner}{thm:main-inner-ii}, completing the proof of \Cref{thm:main-inner}.

\subsection{Counting by product of ramified primes}
\label{subsn:prod-ram}

We now explain how to adapt the proof of \Cref{thm:main-inner} in order to show \Cref{thm:main-inner-prod-ram}.
If $L/K$ is an inner Galois extension of $K$ with center $Z$ corresponding to an element $\lambda \in \Lambda$, then the product~$\ram(L)$ of the primes of $Z$ ramified in $L$ equals the following quantity $\ram(\lambda)$, computed in terms of the map $\lambda$ alone:
\[
  \ram(\lambda)
  \coloneqq
  \prod_{\substack{p\textnormal{ prime of }Z\\\lambda(p)\neq0}} \Nm{p}.
\]
Now, consider the following Dirichlet series:
\[
  f^*_{S,\xi,\tau}(s)
  =
  \sum_{\lambda\in\Lambda_{S,\xi,\tau}}
    \ram(\lambda)^{-s}.
\]
Like in \Cref{ssn:key-dirichlet}, we rewrite the Dirichlet series as:
\begin{align*}
  f^*_{S,\xi,\tau}(s)
  &=
  \ram(\xi)^{-s}
  \sum_{
    \text{(as in \customref{Eq. (\ref*{eqn:f-i})}{eqn:f-i})}
  }
    \,\,\,\,
    \prod_{\substack{p\in\Pm\setminus S\\\lambda(p)\neq0}}
      \,\,
      \Nm{p}^{-s}
  \\&
  =
  \frac{\ram(\xi)^{-s}}M
  \,
  \sum_{\charindex=0}^{M-1}
    \e\leftl(
      \frac{\charindex\sigma_\xi}M
    \rightr)
    \prod_{p\in\Pm\setminus S}
      \mathbf{f}^*_{\tau,\charindex,p}(s).
\end{align*}
where we have defined:
\begin{align*}
  \mathbf{f}^*_{\tau,\charindex,p}(s)
  \coloneqq
  1
  +
  \sum_{\substack{
    0 \neq \lambda \in \Z/M\Z \\
    \eta_{\tau, p} \mid \lambda
  }}
    \e\leftl(
      \frac{\charindex\lambda}M
    \rightr)
    \Nm{p}^{-s}.
\end{align*}
The Euler factor $\mathbf{f}^*_{\tau,\charindex,p}(s)$ also equals 
$
  1
  +
  \psi_{\tau,\charindex}^*(\Frob(p))\Nm{p}^{-s}
$,
where the class function $\psi^*_{\tau,\charindex} : G \to \Z$ is defined as follows:
\[
  \psi^*_{\tau,\charindex}(g)
  \coloneqq
  \begin{cases}
    \gcd\leftl(
      j \cdot \cycgcd(g), \,
      \frac M\tau
    \rightr)
    -1
    &
    \textnormal{if }
    \gcd\leftl(
      j \cdot \cycgcd(g), \,
      \frac M\tau
    \rightr)
    \textnormal{ divides }
    \charindex \\
    -1 &\textnormal{otherwise}.
  \end{cases}
\]
The average of $\psi^*_{\tau,\charindex}$ is largest for $\tau=1$, $\charindex=0$, where it equals:
\[
  \avg\leftl(
    \psi^*_{1,0}(g)
  \rightr)
  =
  j
  \left(
    \frac1{\card{G}}
    \sum_{g\in G}
      \cycgcd(g)
  \right)
  -
  1
  =
  b^*.
\]
The rest of the proof of \Cref{thm:main-inner-prod-ram} goes through exactly as in Subsections \ref{ssn:analytic-props} to \ref{ssn:restriction-divalg}.
In the argument, one uses the fact that $b^* > 0$ (shown in \Cref{rmk:bstar-positive}) in order to see that $f_{S,\xi,\tau}(s)$ has a pole of positive order.

\subsection{Criteria for existence of an extension}
\label{subsn:criteria-existence}

By \Cref{thm:ext-vs-map}, the existence of an inner Galois extension $L/K$ of degree $n=dj^2$ with center $Z(L)=Z$ is equivalent to the existence of a map $\lambda:\Pm\rightarrow\Z/M\Z$ in $\Lambda$, and the existence of such an extension which is a division algebra amounts to the existence of a map $\lambda \in \Lambda'$.
The following lemma shows (when~$S$ is taken to be the empty set) that these conditions can be checked by considering only the finite set $\Pex$ of exceptional primes, and the finitely many maps $\Pex \to \Z/M\Z$:

\begin{theorem}
  \label{thm:criteria-existence}
  Let $S$ be a finite set of places of $Z$ and $\xi$ be a map $S \to \Z/M\Z$ satisfying conditions \ref{def:Lambda:complex}--\ref{def:Lambda:div} of \Cref{def:Lambda}.
  The existence of an inner Galois extension $L/K$ of degree $n=dj^2$ with center $Z(L)=Z$ and whose local invariants at places $v \in S$ are given by $\frac {\xi(v)} M$ is equivalent to the existence of a map $\lambda:\Pex\rightarrow\Z/M\Z$ coinciding with $\xi$ on $S \cap \Pex$, satisfying conditions \ref{def:Lambda:complex}--\ref{def:Lambda:div} of \Cref{def:Lambda}, and such that the following condition holds:
  \begin{enumerate}[label=\textnormal{(IV')}]
    \item
      \label{def:Lambda:sumfin}
      $
        \frac {d m} U
        \mid
        \sum_{v\in \Pex}
          \lambda(v)
      $.
  \end{enumerate}
  The existence of an extension as above which is also a division algebra is equivalent to the existence of a map $\lambda:\Pex\rightarrow\Z/M\Z$ as above which moreover satisfies:
  \begin{enumerate}[label=\textnormal{(V')}] 
    \item
      \label{def:Lambda:gcdfin}
      $
        \gcd\bigl(
          \frac {d m} U,
          \gcd_{v\in \Pex}
            \lambda(v)
        \bigr)
        =
        1
      $.
  \end{enumerate}
\end{theorem}

\begin{proof}
  \hfill
  \begin{description}
    \item[($\Rightarrow$)]
      We first assume that there is an inner Galois extension $L/K$ as above.
      By \Cref{thm:ext-vs-map}, it corresponds to a map $\lambda:\Pm\rightarrow\Z/M\Z$ in $\Lambda$, coinciding with $\xi$ on $S$.
      Its restriction to $\Pex$ clearly still satisfies \ref{def:Lambda:complex}--\ref{def:Lambda:div} at places $v \in \Pex$, and coincides with $\xi$ on $S \cap \Pex$.

      Let $v \in \Pm \setminus \Pex$.
      By definition of $U$,
      $
        \frac {d m} U
        =
        \frac {d m} {\lcm_{g \in G} \,\cycgcd(g)}
      $
      divides
      $
        \frac{dm}{\cycgcd(\Frob(v))}
      $,
      which in turn divides~$\lambda(v)$ by \Cref{lem:csa-inclusion-unramified}.
      Hence, $\frac {d m} U$ divides $\sum_{p \in \Pm \setminus \Pex} \lambda(v)$.

      By condition \ref{def:Lambda:sum} of \Cref{def:Lambda}, we have $\sum_{v\in\Pm} \lambda(v) = 0$.
      It follows that $\frac {d m} U \mid \sum_{v \in \Pex}\lambda(v)$, so the restriction of $\lambda$ to $\Pex$ satisfies \ref{def:Lambda:sumfin}.

      If $L$ is a division algebra, then $\lambda\in\Lambda'$, so by condition \ref{def:Lambda:gcd} of \Cref{def:Lambda}, we have $\gcd_{v \in \Pm} \lambda(v) = 1$.
      Since $\frac {d m} U \mid \lambda(v)$ for $v\in\Pm\setminus\Pex$, it follows that $\gcd\bigl(\frac {d m} U, \, \gcd_{v\in \Pex}\lambda(v)\bigr) = 1$, so the restriction of $\lambda$ to $\Pex$ satisfies \ref{def:Lambda:gcdfin}.

    \item[($\Leftarrow$)]
      Conversely, assume we have a map $\lambda:\Pex\rightarrow\Z/M\Z$ as above, satisfying \ref{def:Lambda:complex}--\ref{def:Lambda:div} and \ref{def:Lambda:sumfin} and coinciding with $\xi$ on $S \cap \Pex$.
      We explain how to extend it to a map $\lambda:\Pm\rightarrow\Z/M\Z$ satisfying \ref{def:Lambda:sum} and coinciding with $\xi$ on $S$.
      First, by \v Cebotarev's density theorem, the equality $\frac {d m} U = \gcd_{g\in G} \frac{dm}{\cycgcd(g)}$ implies:
      \begin{equation}
        \label{eqn:dm/U-cebotarev}
        \frac
          {d m}
          U
        =
        \gcd_{p \not\in S \cup \Pex}
          \frac
            {dm}
            {\cycgcd(\Frob(p))}.
      \end{equation}

      By hypothesis, $\frac {d m} U$ divides $\sum_{v\in \Pex}\lambda(v)$.
      Moreover, $\frac {d m} U$ divides $\xi(v)$ for every $v \in S \setminus \Pex$ because~$\xi$ satisfies \ref{def:Lambda:div}.
      Extend $\lambda$ to a map $S \cup \Pex \to \Z/M\Z$ by putting $\lambda(v) = \xi(v)$ if $v \in S \setminus \Pex$.
      Then, $\frac {d m} U$ divides $\sum_{v\in S \cup \Pex} \lambda(v)$.
      By \Cref{eqn:dm/U-cebotarev} and Bézout's identity,
      there is a finite set $S_1$ of primes $p \notin S \cup \Pex$ and integers $\tilde\lambda(p)$ for each $p \in S_1$, such that $\tilde\lambda(p)$ is divisible by $\frac {dm} {\cycgcd(\Frob(p))}$ and such that:
      \begin{equation}
        \label{eqn:sum-lambda-zero}
        \sum_{v \in S \cup \Pex}
          \lambda(v)
        +
        \sum_{v \in S_1}
          \tilde\lambda(v)
        = 0.
      \end{equation}
      Extend $\lambda$ to $\Pm$ by letting $\lambda(p) = \tilde\lambda(p)$ if $p \in S_1$, and $\lambda(v)=0$ for places $v\notin (S \cup \Pex) \sqcup S_1$.
      We have extended $\lambda$ into a map $\lambda:\Pm\rightarrow\Z/M\Z$ coinciding with $\xi$ on $S$; by \Cref{lem:csa-inclusion-unramified} and \Cref{eqn:sum-lambda-zero}, this map satisfies conditions \ref{def:Lambda:complex}--\ref{def:Lambda:sum} and hence lies in~$\Lambda$.
      This gives us an inner Galois extension $L/K$ as needed.

      Now assume also that the original map $\lambda:\Pex\rightarrow\Z/M\Z$ satisfies~\ref{def:Lambda:gcdfin}.
      It follows from \Cref{eqn:dm/U-cebotarev} that there is a finite set $S_2$ of places $p \notin S \cup \Pex$ such that $\frac {dm} U = \gcd_{p\in S_2} \frac {dm}{\cycgcd(\Frob(p))}$.
      Add these places to $S$ and extend $\xi$ by letting $\xi(p) = \frac {dm} {\cycgcd(\Frob(p))}$ if $p \in S_2$.
      Now, use the procedure from the previous paragraph to extend $\lambda$ to a map $\Pm \to \Z/M\Z$ coinciding with~$\xi$ on $S$ and satisfying \ref{def:Lambda:complex}--\ref{def:Lambda:sum}.
      Then $\gcd_{v \in \Pm} \lambda(v)$ divides
      $
        \gcd\bigl(
          \gcd_{v \in \Pex} \lambda(v), \,
          \gcd_{v \in S_2} \lambda(v)
        \bigr)
        =
        \gcd\bigl(
          \gcd_{v \in \Pex} \lambda(v), \,
          \frac {d m} {\cycgcd(\Frob(p))}
        \bigr)
        =
        \gcd\bigl(
          \gcd_{v \in \Pex} \lambda(v), \,
          \frac {d m} U
        \bigr)
      $,
      which is $1$ by hypothesis~\ref{def:Lambda:gcdfin}.
      This shows that the map $\lambda:\Pm\rightarrow\Z/M\Z$ satisfies \ref{def:Lambda:gcd}.
      By \Cref{thm:ext-vs-map}, the corresponding extension $L/K$ is therefore a division algebra as required.
      \qedhere
  \end{description}
\end{proof}

\begin{remark}
  There are indeed situations where a map $\lambda$ like in \Cref{thm:criteria-existence} does not exist, even when $S$ is empty.
  For instance, assume that $F/Z$ is a nontrivial Galois extension and that $m$ does not divide $j$.
  Let $v$ be a prime of $Z$ completely split in~$F$, so that condition \ref{def:Lambda:div} rewrites as ``for all places~$w|v$ of $F$, $\lambda(v) = d j \kappa_w$ mod $dm$''.
  It is then possible to construct the central simple $F$-algebra $K$ of dimension~$m^2$ so that the values of $\kappa_w$ for different primes $w$ above $v$ force $\lambda(v)$ to take contradictory values modulo~$dm$.
  For example, choose $\kappa_w$ to be $0$ for some place $w|v$, and $1$ for some other place~$w|v$.
  Then, $\lambda(v)$ must be congruent to both $0$ and $dj$ modulo $dm$, which is impossible as~$m$ does not divide $j$.
\end{remark}


Note also the following corollary:

\begin{corollary}
  If there is an inner Galois extension of $K$ of degree $n$ with center $Z$ which is a division algebra, then $m$ and $j$ are coprime.
\end{corollary}

\begin{proof}
  By \Cref{thm:criteria-existence}, the hypothesis implies the existence of a map $\lambda : \Pex \to \Z/M\Z$ satisfying conditions \ref{def:Lambda:complex}--\ref{def:Lambda:div} of \Cref{def:Lambda} as well as condition~\ref{def:Lambda:gcdfin}.
  The integer $d\gcd(m,j)$ divides both~$dm$ and $dj$.
  Since $\lambda$ satisfies condition \ref{def:Lambda:div} of \Cref{def:Lambda}, this implies that $d\gcd(m,j)$ divides~$d_w \lambda(v)$ for all places $w$ of $F$ lying above a place $v \in \Pex$.
  Summing over all places~$w$ above a fixed place $v \in \Pex$, we get $d\gcd(m,j) \mid d\lambda(v)$, i.e., $\gcd(m,j) \mid \lambda(v)$.
  This holds for all places $v \in \Pex$, and thus $\gcd(m,j)$ divides $\gcd_{v \in \Pex} \lambda(v)$.
  On the other hand, $\gcd(m,j) \mid m \mid \frac {d m} U$.
  Hence, $\gcd(m,j)$ divides
  $
    \gcd\Bigl(
      \frac {d m} U, \,
      \gcd_{v \in \Pex} \lambda(v)
    \Bigr)
  $, which equals $1$ by condition~\ref{def:Lambda:gcdfin}.
  We conclude that $\gcd(m,j) = 1$ as claimed.
\end{proof}

\section{Outer extensions}
\label{sn:outer}

In this section, we study the distribution of outer extensions $L$ of a given (finite-dimensional) simple $\Q$-algebra $K$.
We also discuss more specific questions, considering the case of outer Galois extensions with a given finite Galois group or restricting our attention to division algebras.
In all cases, our main results relate these problems to questions concerning the distribution of ordinary commutative field extensions.
Those questions are open in most cases.

In \Cref{subsn:prelim-outer}, we prove general facts concerning outer extensions, notably \Cref{thm:charac-outer} which states that every outer extension of $K$ is the tensor product of $K$ with a field extension of its center~$Z(K)$.
In \Cref{subsn:tensor-products}, we give criteria for such a tensor product to be a division algebra and express $d(L/K)$ in terms of the relative discriminant of the field extension $Z(L)/Z(K)$.
Finally, in \Cref{subsn:counting-outer}, we use these descriptions to relate the problem of counting outer extensions of $K$ with that of counting field extensions of $Z(K)$.

\subsection{General facts about outer extensions}
\label{subsn:prelim-outer}


\subsubsection{Groups of inner automorphisms are either trivial or infinite.}

We first prove \Cref{prop:quotient-of-units-not-finite} and \Cref{cor:outer-if-aut-finite}, which imply that every extension whose automorphism group is finite is an outer extension, assuming that its center is infinite.
For division algebras, this is well-known (see \cite[Théorème]{deschamps-petites}).
\Cref{prop:quotient-of-units-not-finite} is also used in our proof of the characterization \Cref{thm:charac-outer}.

\begin{proposition}
  \label{prop:quotient-of-units-not-finite}
  Let $A$ be a finite-dimensional algebra over an infinite field $F$.
  If the group $A^\times/F^\times$ is finite, then $A = F$.
\end{proposition}

\begin{proof}
  Write $A^\times = a_1 F^\times \sqcup \cdots \sqcup a_r F^\times$.
  Then, $A$ is covered by the subsets $A\setminus A^\times$ and $a_1 F, \dots, a_r F$.
  Each of these is an algebraic subset of the $F$-vector space~$A$:
  The set $A\setminus A^\times$ is defined by the equation $\textnormal{Nm}_{A/F}(x)=0$, which is polynomial in the coordinates of $x$, and the sets $a_1 F,\dots,a_r F$ are linear subspaces.
  If $A\neq F$, then the $F$-vector space $A$ is at least two-dimensional, so $A\setminus A^\times,a_1 F,\dots,a_r F$ are properly contained in $A$.
  However, it is well-known that a finite-dimensional vector space over an infinite field cannot be covered by finitely many algebraic proper subsets. (See \cite[p.\ 228]{cohn-basic-algebra}.)
\end{proof}


\begin{corollary}
  \label{cor:outer-if-aut-finite}
  Let $L/K$ be an extension of simple algebras.
  Assume that the inner automorphism group of $L/K$ is finite and that the field $Z(L)$ is infinite.
  Then, $L/K$ is an outer extension.
\end{corollary}

\begin{proof}
  By \Cref{prop:quotient-of-units-not-finite}, the finiteness of $\Inn(L/K) \simeq \Cent_L(K)^\times/Z(L)^\times$ implies that $\Cent_L(K) = Z(L)$ and thus $\Inn(L/K) = 1$.
  Therefore, the extension $L/K$ is outer.
\end{proof}

(The assumption that $Z(L)$ be infinite is crucial, as the example $L=\mathfrak M_d(\mathbb F_q)$, $K=\mathbb F_q$ shows.)


\subsubsection{The Deschamps-Legrand descent theorem.}
\label{sssn:descent}

We state and prove a generalized version of a theorem of Deschamps and Legrand \cite[Corollaire~2]{desleg}, which is an ``outer equivalent'' of \Cref{cor:descent-inner}.
The original result does not deal with the non-Galois case, and only treats the case of division algebras.
This theorem gives a concrete description of extensions of $K$, allowing to parametrize them in terms of field extensions of $Z(K)$ (see \Cref{cor:outer-is-like-commutative}).

\begin{theorem}
  \label{thm:charac-outer}
  Let $L/K$ be an extension of simple $\Q$-algebras.
  The following are equivalent:
  \begin{enumerate}[label=(\roman*)]
    \item
      \label{thm:charac-outer-i}
      $L/K$ is outer;
    \item
      \label{thm:charac-outer-ii}
      $\Cent_L(K) = Z(L)$;
    \item
      \label{thm:charac-outer-iv}
      $Z(K)$ is contained in $Z(L)$, and $L$ is generated by $K$ and $Z(L)$;
    \item
      \label{thm:charac-outer-v}
      $Z(K)$ is contained in $Z(L)$, and $L$ is isomorphic to the tensor product $Z(L) \otimes_{Z(K)} K$ as an extension of $K$;
    \item
      \label{thm:charac-outer-vi}
      $Z(K)$ is contained in $Z(L)$, and the restriction map $\Aut(L/K) \to \Aut(Z(L)/Z(K))$ is bijective.
  \end{enumerate}
\end{theorem}

\begin{proof}
  ~
  \begin{itemize}
    \item[
      \ref{thm:charac-outer-i}
      $\Rightarrow$
      \ref{thm:charac-outer-ii}
    ]
      Since $L/K$ is outer, the group $\Inn(L/K) \simeq \Cent_L(K)^{\times}/Z(L)^{\times}$ is trivial and thus $\Cent_L(K)^{\times} = Z(L)^{\times}$.
      Now, \Cref{prop:quotient-of-units-not-finite} directly implies that $\Cent_L(K) = Z(L)$.

    \item[
      \ref{thm:charac-outer-ii}
      $\Rightarrow$
      \ref{thm:charac-outer-iv}
    ]
      We have
      $
        Z(K)
        =
        K \cap \Cent_L(K)
        \overset{\text{\ref{thm:charac-outer-ii}}}=
        K \cap Z(L)
      $, so $Z(K)$ is contained in $Z(L)$.
      Let~$L^{\Delta}$ be the subalgebra of $L$ generated jointly by $Z(L)$ and $K$.
      We have:
      \begin{align*}
        L^{\Delta}
        & =
        \Cent_L\Big(
          \Cent_L\big(
            L^{\Delta}
          \big)
        \Big)
        & \text{by \cite[\href{https://stacks.math.columbia.edu/tag/074T}{Theorem~074T~(3)}]{stacks-project}}
        \\
        & =
        \Cent_L\Big(
          \Cent_L\big(
            Z(L)
          \big)
          \cap
          \Cent_L(K)
        \Big)
        & \text{because $L^{\Delta}$ is generated by $Z(L)$ and $K$}
        \\
        & =
        \Cent_L\Big(
          L \cap \Cent_L(K)
        \Big)
        \\
        & =
        \Cent_L(Z(L))
        &
        \text{by \ref{thm:charac-outer-ii}}
        \\
        & =
        L.
      \end{align*}
      Therefore, $L$ is generated jointly by $Z(L)$ and $K$.
    \item[
      \ref{thm:charac-outer-iv}
      $\Rightarrow$
      \ref{thm:charac-outer-v}
    ]
      By \ref{thm:charac-outer-iv}, the algebras $Z(L)$ and $K$ generate $L$, and thus the tensor product $Z(L) \otimes_{Z(K)} K$ surjects onto~$L$ via $x\otimes y\mapsto xy$.
      Since $L \neq 1$ and the algebra $Z(L) \otimes_{Z(K)} K$ is simple by \cite[\href{https://stacks.math.columbia.edu/tag/074F}{Lemma~074F}]{stacks-project}, this surjection is an isomorphism.
      Therefore, $L \simeq Z(L) \otimes_{Z(K)} K$ as claimed.

    \item[
      \ref{thm:charac-outer-v}
      $\Rightarrow$
      \ref{thm:charac-outer-vi}
    ]
      The restriction map is well-defined because automorphisms of $L$ preserve the center.
      By functoriality of the tensor product, we obtain a map $\Aut(Z(L)/Z(K)) \to \Aut(Z(L) \otimes_{Z(K)} K/K) \overset{\ref{thm:charac-outer-v}}{\simeq} \Aut(L/K)$.
      That map is an inverse of the restriction map, proving its bijectivity.

    \item[
      \ref{thm:charac-outer-vi}
      $\Rightarrow$
      \ref{thm:charac-outer-i}
    ]
      By \ref{thm:charac-outer-vi}, automorphisms of $L/K$ are determined by their restriction to $Z(L)$.
      However, inner automorphisms of $L$ act trivially on $Z(L)$.
      Therefore, the extension $L/K$ has no nontrivial inner automorphisms and thus is outer.
      \qedhere
  \end{itemize}
\end{proof}

We rephrase \Cref{thm:charac-outer} to emphasize its importance for parametrizing outer extensions of~$K$:

\begin{corollary}
  \label{cor:outer-is-like-commutative}
  Let $F$ be a number field and let $K$ be a central simple $F$-algebra.

  {
    \raggedright
    \setlength{\tabcolsep}{0em}
    \begin{tabular}{>{\raggedleft}p{.5\textwidth}>{\centering}p{.10\textwidth}>{\raggedright\arraybackslash}p{.4\textwidth}}
      \multicolumn{3}{l}{
        There is a bijective correspondence:
      }
      \\
      $
        \left\{
          \text{outer extensions } L/K
        \right\}/\cong
      $
      &
      $\overset\sim\longleftrightarrow$
      &
      $
        \left\{
          \text{field extensions } F'/F
        \right\}/\cong
      $
      \\
      \multicolumn{3}{l}{
        given by:
      }
      \\
      $
        L
      $
      &
      $\mapsto$
      &
      $
        Z(L)
      $
      \\
      $
        F' \otimes_F K
      $
      &
      $\otspam$
      &
      $
        F'
      $
      \\
      \multicolumn{3}{l}{
        Moreover, via this correspondence:
      }
      \\
      $
        [L:K]
      $
      &
        $=$
      &
      $
        [F':F]
      $
      \\
      $
        \Aut(L/K)
      $
      &
      $\simeq$
      &
      $
        \Aut(F'/F)
      $
      \\
      $L/K$ is Galois
      &
      $\Leftrightarrow$
      &
      $F'/F$ is Galois.
    \end{tabular}
  }
\end{corollary}

\begin{proof}
  Everything follows directly from \Cref{thm:charac-outer} apart from the last equivalence.
  By \Cref{thm:charac-outer} [\ref{thm:charac-outer-i} $\Rightarrow$ \ref{thm:charac-outer-v}, \ref{thm:charac-outer-vi}], we have:
  \[
    L^{\Aut(L/K)} = (Z(L) \otimes_{Z(K)} K)^{\Aut(Z(L)/Z(K))}
  \]
  where $\Aut(Z(L)/Z(K))$ only acts on the factor $Z(L)$, so that:
  \begin{equation}
    \label{eqn:fixedfield-outer}
    L^{\Aut(L/K)} = Z(L)^{\Aut(Z(L)/Z(K))} \otimes_{Z(K)} K.
  \end{equation}
  Thus, $L^{\Aut(L/K)} = K$ is equivalent to $Z(L)^{\Aut(Z(L)/Z(K))}=Z(K)$.
\end{proof}

\begin{remark}
  \Cref{cor:outer-is-like-commutative} implies a ``noncommutative Hilbert 90 theorem'' like in \cite[Proposition~26.2]{deschamps}.
  Indeed, consider an outer Galois extension $L/K$ of simple algebras.
  Combining \Cref{cor:outer-is-like-commutative} with the result of \cite[Chap.~X, §10, Exercise~2]{serrecl}, we obtain:
  \[
    H^1\Big(
      \Gal(L/K),
      L^{\times}
    \Big)
    =
    H^1\Big(
      \Gal\big(
        Z(L)/Z(K)
      \big),
      \big(
        K \otimes_{Z(K)} Z(L)
      \big)^{\times}
    \Big)
    = 1.
  \]
\end{remark}

\subsection{Computing relative discriminants of outer extensions}
\label{subsn:tensor-products}

Let $F$ be a number field and $K$ be a central division $F$-algebra of dimension $m^2$.
If $v$ is a place of~$F$, let~$\kappa_v$ be the element of $\Z/m\Z$ such that the local invariant of $K$ at $v$ is given by $\frac {\kappa_v} m$.
We denote by~$S$ the finite set of places $v$ of $F$ for which $\kappa_v \neq 0$.

In this section, we take a closer look at tensor products of $K$ with field extensions $F'/F$, which by \Cref{cor:outer-is-like-commutative} are all the outer extensions of $K$.
The main result is \Cref{thm:tensor-disc-divalg}, which relates the ``generalized relative discriminant'' $d(L/K)$ when $L = K \otimes_F F'$ to the relative discriminant $d(F'/F)$, and characterizes situations where $L$ is a division algebra.

Let $d \in \N$.
We introduce the set $\mathcal{E}_d$ of tuples $(E(v))_{v \in S}$ where $E(v)$ is an étale $F_v$-algebra of dimension $d$ for all $v \in S$.
This set $\mathcal{E}_d$ is finite.%
\footnote{
  Here, complications arise if the base field is a function field over a finite field instead of a number field, as its completions may have infinitely many extensions of a given degree and thus the naive analogue of $\mathcal{E}_d$ is not always finite.
  Then, when taking cardinalities in \Cref{eqn:bijec-ed}, the sum on the right is not finite.
}
To each field extension $F'/F$ of degree $d$ corresponds a tuple $(F' \otimes_F F_v)_{v \in S} \in \mathcal{E}_d$.
For $E \in \mathcal{E}_d$, we define:
\begin{equation}
  \label{defn:deltaE}
  \delta(E)
  \coloneqq
  \prod_{\substack{
    p \in S \\
    \text{ prime of }F
  }}
    \Nm{p}^{\delta_p(E)}
  \end{equation}
where
\begin{equation}
  \label{defn:deltapE}
  \delta_p(E)
  \coloneqq
  m d (m - \gcd(m, \, \kappa_p))
  -
  m
  \sum_{\substack{
    \text{field } E' \\
    \text{factor of } E(p)
  }}
    f(E'/F_v)
    (
      m - \gcd(m, \, [E':F_v] \kappa_p)
    )
  .
\end{equation}

We also define the following subset of $\mathcal{E}_d$:
\[
  \mathcal{E}'_d
  \coloneqq
  \left\{
    (E(v))_{v \in S} \in \mathcal{E}_d
    \verti
    \text{
      the elements
      $
        \Big(
          [E':F_v] \kappa_v
        \Big)_{\begin{lsubstack}
          v \in S\\
          E' \text{ factor of } E(v)
        \end{lsubstack}}
    $
    generate $\Z/m\Z$}
  \right\}.
\]

\begin{theorem}
  \label{thm:tensor-disc-divalg}
  Let $F'/F$ be a field extension of degree $d$ and let $L/K$ be an outer extension, associated to each other via the bijection of \Cref{cor:outer-is-like-commutative} (i.e., $F'=Z(L)$ and $L = F' \otimes_F K$).
  Let $E = (F' \otimes_F F_v)_{v \in S} \in \mathcal{E}_d$.
  Then:
  \begin{enumerate}[label=(\roman*)]
    \item
    The number $d(L/K)$ introduced in \Cref{def:discriminant} satisfies $d(L/K) = \delta(E)^{-1} \cdot d(F'/F)^{m^2}$;
    \item
      $L$ is a division algebra if and only if $E \in \mathcal{E}'_d$.
  \end{enumerate}
\end{theorem}


\begin{proof}
  By \cite[(31.9)]{reiner}, the local invariant of $L$ at a place $w$ of $F'$, lying above a place $v$ of $F$, is given by:
  \begin{equation}
    \label{eqn:localinv-tensor}
    \inv{L_w}
    =
    \inv{
      (K \otimes_F F')_w
    }
    =
    [F'_w:F_v]
    \cdot
    \inv{K_v}
    =
    \frac{
      [F'_w:F_v]\kappa_v
    }
    m.
  \end{equation}
  \begin{enumerate}[label=(\roman*)]
    \item
      We have:
      \begin{align*}
        d(L/F')
        &
        =
        \prod_{p \text{ prime of }F}
          \,
          \prod_{q|p \text{ prime of } F'}
            \,
            \Nm{q}^{
              m\Big(
                m
                -
                \gcd(m, \, [F'_q:F_p] \kappa_p)
              \Big)
            }
        &
        \text{by \Cref{eqn:discr-csa,eqn:localinv-tensor}}
        \\
        &
        =
        \prod_{p \text{ prime of }F}
          \,
          \prod_{q|p \text{ prime of } F'}
            \,
            \Nm{p}^{
              mf(q|p)\Big(
                m
                -
                \gcd(m, \, [F'_q:F_p] \kappa_p)
              \Big)
            }
        \\
        &
        =
        \prod_{\substack{p \in S\\ \text{ prime of }F}}
          \Nm{p}^{
            m
            \sum_{q|p}
              f(q|p)\Big(
                m
                -
                \gcd(m, \, [F'_q:F_p] \kappa_p)
              \Big)
          }
        \\
        &
        =
        \prod_{\substack{p \in S\\ \text{ prime of }F}}
          \Nm{p}^{
            md\Big(m-\gcd(m, \kappa_p)\Big) - \delta_p(E)
          }
        &
        \text{by \Cref{defn:deltapE}}
        \\
        &
        =
        d(K/F)^d
        \cdot
        \delta(E)^{-1}
        &
        \text{by \Cref{eqn:discr-csa,defn:deltaE}}.
      \end{align*}
      Finally, using \Cref{prop:reldisc}:
      \[
        d(L/K)
        =
        \frac{d(L/F') \cdot d(F'/F)^{m^2}}{d(K/F)^d}
        = \delta(E)^{-1} \cdot d(F'/F)^{m^2}.
      \]
    \item
      By definition, $L$ is a division algebra if and only if its index is $m$.
      Since index and exponent coincide for central simple algebras over number fields, this is equivalent to the condition that the invariants $\inv{L_w} = \frac{[F_w':F_v]\kappa_v}{m}$ have least common denominator $m$, which amounts to the numerators $[F_w':F_v]\kappa_v$ generating $\Z/m\Z$.
      The places $v\notin S$, with $\kappa_v=0$, do not contribute.
      For $v\in S$, the fields $F_w'$ with $w \mid v$ are exactly the factors $E'$ of $E(v)=F'\otimes_F F_v$.
      Thus, $L$ is a division algebra if and only if $E \in \mathcal{E}'_d$.
      \qedhere
  \end{enumerate}
\end{proof}



\begin{remark}
  The number $\delta(E)$ is the $m$-th power of an integer and divides $d(K/F)^d$.
  Indeed, for all primes $p \in S$, we have:
  \begin{align*}
    \delta_p(E)
    &
    =
    m
    \sum_{\substack{
      \text{field } E' \\
      \text{factor of } E(p)
    }}
      f(E'/F_v)
      \left[
        e(E'/F_v)
        (m - \gcd(m, \, \kappa_p))
        -
        (
          m - \gcd(m, \, [E':F_v] \kappa_p)
        )
      \right]
    \\
    &
    \geq
    m
    \sum_{\substack{
      \text{field } E' \\
      \text{factor of } E(p)
    }}
      f(E'/F_v)
      \left[
        e(E'/F_v)
        (m - \gcd(m, \, \kappa_p))
        -
        (
          m - \gcd(m, \, \kappa_p)
        )
      \right]
    \\
    &
    =
    m
    \sum_{\substack{
      \text{field } E' \\
      \text{factor of } E(p)
    }}
      f(E'/F_v)
      (e(E'/F_v) - 1)
      (m - \gcd(m, \, \kappa_p))
    \\
    & \geq 0
  \end{align*}
  and thus $\delta(E)$ is an integer.
  Since the integers $\delta_p(E)$ are multiples of $m$, $\delta(E)$ is an $m$-th power.
  Finally, $\delta_p(E) \leq md(m-\gcd(m,\kappa_p))$ so $\delta(E)$ divides $d(K/F)^d$ (compare with \Cref{eqn:discr-csa}).
\end{remark}

\begin{remark}
  Assume $K$ is a division algebra and $d$ is coprime to $m$.
  Then, $\mathcal{E}_d = \mathcal{E}'_d$ (i.e., $K \otimes_F F'$ is a division algebra for all field extensions $F'/F$ of degree $d$).
  Indeed, consider an element $E \in \mathcal{E}_d$.
  For each place $v \in S$, $\gcd_{E' \text{ factor of } E(v)} [E':F_v]$ divides $\sum_{E' \text{ factor of } E(v)} [E':F_v] = [E(v):F_v] = d$.
  Since~$d$ and $m$ are coprime, this implies
  $
  \gcd_{v \in S}
    \gcd_{
      E' \text{ factor of } E(v)
    }
      [E':F_v] \kappa_v
  =
  \gcd_{v \in S}
    \kappa_v$,
  which equals $1$ because $K$ is a division algebra unramified outside $S$.
  We have shown $E \in \mathcal{E}'_d$.
\end{remark}

\subsection{Counting outer extensions.}
\label{subsn:counting-outer}

All notations are as in \Cref{subsn:tensor-products}.
By \Cref{thm:tensor-disc-divalg}, the bijection of \Cref{cor:outer-is-like-commutative} restricts to a bijection:
\begin{equation}
  \label{eqn:bijec-ed}
  \left\lbrace
  \begin{matrix}
    \text{isomorphism classes of}\\
    \text{outer extensions $L/K$}\\
    \text{such that $d(L/K) \leq X$}
  \end{matrix}
  \right\rbrace
  \,\,
  \overset\sim\longleftrightarrow
  \,\,
  \bigsqcup_{E \in \mathcal{E}_d}
    \left\lbrace
    \begin{matrix}
      \text{isomorphism classes of}\\
      \text{field extensions $F'/F$}\\
      \text{such that $d(F'/F) \leq (\delta(E) X)^{1/m^2}$}\\
      \text{and $F' \otimes_F F_v \simeq E(v)$ for all $v \in S$}
    \end{matrix}
    \right\rbrace.
\end{equation}
This bijection can be restricted to outer extensions which are division algebras by considering only tuples $E \in \mathcal{E}'_d$ on the right-hand side.
We can therefore relate the counting function for outer extensions of~$K$ to a finite sum of counting functions for field extensions of $F$ with fixed local behaviors above~$S$.

For example, the results of \cite[Sections~2.4 and~2.5]{wood-probabilities-of-local-behaviors} on Malle's conjecture for abelian groups imply:
\begin{corollary}
  Let $G$ be a finite abelian group.
  Let $u$ be the smallest prime divisor of $\card{G}$ and let $r$ be the number of elements of order $u$ in $G$.
  \begin{enumerate}[label=(\roman*)]
    \item
      There is a real number $C>0$ such that the number $N(X)$ of isomorphism classes of Galois extensions $L/K$ with Galois group isomorphic to $G$ (necessarily outer by \Cref{cor:outer-if-aut-finite}) and $d(L/K) \leq X$ satisfies
      \[
        N(X)
        \underset{X\to\infty}\sim
        C
        X^{1/a}(\log X)^{b-1}
      \]
      where $a = m^2\card{G}\left(1-\frac1u\right)$ and $b = \frac{r}{[F(\zeta_u):F]}$.
    \item
      Assuming that $K$ is a division algebra, the same holds if we restrict to extensions $L/K$ which are division algebras (with a possibly smaller constant $C$).
  \end{enumerate}
\end{corollary}

Similarly, \cite[Theorem~3]{bsw-counting-number-fields} implies:
\begin{corollary}
  Let $n\in\{2,3,4,5\}$.
  \begin{enumerate}[label=(\roman*)]
    \item
      There is a real number $C>0$ such that the number $N(X)$ of isomorphism classes of outer extensions $L/K$ of degree $n$ with $d(L/K)\leq X$ satisfies
      \[
        N(X)
        \underset{X\to\infty}\sim
        C
        X^{1/m^2}.
      \]
    \item
      Assuming that $K$ is a division algebra, the same holds if we restrict to extensions $L/K$ which are division algebras (with a possibly smaller constant $C$).
  \end{enumerate}
\end{corollary}


\section{General extensions}
\label{sn:gen-ext}

In this section, we briefly discuss the possibility of combining the methods of \Cref{sn:outer} with the methods of \Cref{sn:inner} in order to parametrize or count general extensions which are neither inner nor outer.
We focus exclusively on extensions $L/K$ which are division algebras.
The main result is \Cref{thm:general-bijection}, which explains how to uniquely decompose such an extension $L/K$ into an outer extension $F'\otimes_F K / K$ and an inner Galois extension $L/F'\otimes_F K$.
This decomposition can be used ``backwards'' to parametrize extensions~$L/K$.
Moreover, we relate the outer automorphism group of~$L/K$ and the Galois group of~$F'/F$.

In the proof of \Cref{thm:general-bijection}, we make use of \Cref{lem:extension-of-auts} below, which lets one extend automorphisms of a field into automorphisms of simple central algebras over that field.
A proof is given in \cite[Proposition 5.8]{hanke}, where the result is attributed to Deuring.

\begin{lemma}
  \label{lem:extension-of-auts}
  Let $Z$ be a number field and let $L$ be a central simple $Z$-algebra.
  Then, an automorphism $\sigma \in \Aut(Z)$ extends into an automorphism $\tilde\sigma \in \Aut(L)$ if and only if $\inv{L_v} = \inv{L_{\sigma(v)}}$ for every place $v$ of~$L$.
  We say that an automorphism of~$Z$ \emph{preserves $L$} if it satisfies that property.
\end{lemma}

Finally, we state and prove \Cref{thm:general-bijection}:

\begin{theorem}
  \label{thm:general-bijection}
  Let $K$ be a division $\Q$-algebra with center $F$.

  \begin{enumerate}[label=(\roman*)]
    \item
      \label{item:general-bijection-bij}
      We have a bijection between the set of isomorphism classes of extensions $L/K$ that are division algebras and the set of equivalence classes of triples $(F',Z,L)$, where:
      \begin{itemize}
        \item
          $F'/F$ is a finite field extension,
        \item
          $Z$ is a subfield of $F'$ satisfying $F'=Z\cdot F$,
        \item
          $L$ is an extension with center $Z$ of the central simple $F'$-algebra $F'\otimes_F K$ such that $L$ is a division algebra.
          By \Cref{lem:inner-galext-incl-center}, such an extension in inner Galois.
      \end{itemize}

      Here, two triples $(F_1',Z_1,L_1)$ and $(F_2',Z_2,L_2)$ are considered equivalent if there is an $F$-algebra isomorphism $f:F_1'\rightarrow F_2'$ with $f(Z_1)=Z_2$ and a ring isomorphism $g:L_1\rightarrow L_2$ such that $g(x\otimes y)=f(x)\otimes y$ for all $x\in F_1'$ and $y\in K$.
  \end{enumerate}
  Moreover, if an extension $L/K$ corresponds to a triple $(F',Z,L)$ via this bijection, then:
  \begin{enumerate}[resume,label=(\roman*)]
    \item
      \label{item:general-bijection-aut}
      The outer automorphism group $\Out(L/K) \coloneqq \Aut(L/K) / \Inn(L/K)$ of $L/K$ is isomorphic to the group of automorphisms of $F'/F$ sending $Z$ to $Z$ and whose restriction to $Z$ preserves $L$.
    \item
      \label{item:general-bijection-galois}
      The extension $L/K$ is Galois if and only if the field extension $F'/F$ is Galois and every automorphism of $F'/F$ restricts to a well-defined automorphism of $Z$ preserving $L$.
  \end{enumerate}
\end{theorem}
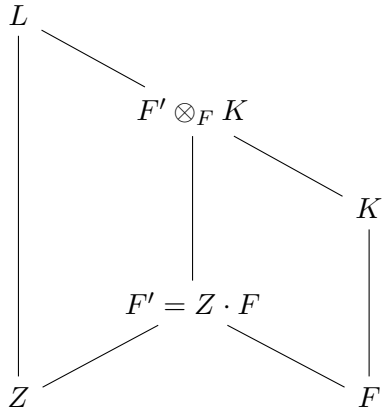
\[
  \begin{tikzcd}
    L \ar[-]{rd} \ar[-]{dddd} \\
    & F'\otimes_F K \ar[-]{rd} \ar[-]{dd} \\
    & & K \ar[-]{dd} \\
    & F' = Z \cdot F \ar[-]{dl} \ar[-]{dr} \\
    Z & & F
  \end{tikzcd}
\]
\begin{proof}
  Any triple $(F',Z,L)$ as above naturally gives rise to an extension $L/K$ which is a division algebra, as~$L$ is an extension of $F'\otimes_F K$ and hence of $K$.
  Equivalent triples by definition give rise to isomorphic extensions of $K$.

  Conversely, consider any extension $L/K$ that is a division algebra.
  To construct the triple $(F',Z,L)$, we first let $Z \coloneqq Z(L)$, and we let $F' \coloneqq Z\cdot F$ be the smallest subring of $L$ containing $Z$ and $F$.
  As a commutative finite-dimensional $\Q$-algebra without zero divisors, $F'$ is a field.
  Since elements of $F'$ commute with those of $K$, we have a $Z$-algebra homomorphism $F'\otimes_F K\rightarrow L$ sending $f\otimes k$ to $f k$, which is injective since $F'\otimes_F K$ is a simple ring by \cite[\href{https://stacks.math.columbia.edu/tag/074F}{Lemma~074F}]{stacks-project}.
  Using this embedding, we can interpret $L$ as an extension of $F'\otimes_F K$.
  This concludes the construction of $(F',Z,L)$.
  
  Consider any isomorphism $g:L_1\rightarrow L_2$ between extensions of $K$ which are division algebras.
  It restricts to an isomorphism $Z(L_1)\rightarrow Z(L_2)$ and fixes $F \subseteq K$.
  Hence, it restricts to an isomorphism $f:Z(L_1)\cdot F\rightarrow Z(L_2)\cdot F$ with $f(Z(L_1))=Z(L_2)$.
  Moreover, $g(x y)=f(x) y$ for all $x\in Z(L_1)\cdot F$ and $y\in K$.
  This implies that isomorphic extensions give rise to equivalent tuples, completing the proof of \ref{item:general-bijection-bij}.
  We leave it to the reader to verify that the maps are inverse to each other.
  
  Let $(F', Z, L)$ be a triple as above.
  Reasoning as in the previous paragraph, we see that we have a group homomorphism:
  \[
    \varphi:
    \Aut(L/K)
    \longrightarrow
    \Big\{
      \sigma \in \Aut(F'/F)
      \,\Big\vert\,
      \sigma(Z)=Z
    \Big\}.
  \]
  Any element of the kernel of $\varphi$ is an automorphism of $L/Z$ and hence an inner automorphism by the Skolem--Noether theorem.
  Conversely, any inner automorphism of $L/K$ fixes all elements of $Z=Z(L)$ and of $F\subseteq K$ and hence fixes all elements of $F'$.
  Therefore, the group homomorphism $\varphi$ has kernel $\Inn(L/K)$ and thus induces an injective homomorphism:
  \[
    \tilde\varphi:
    \Out(L/K)
    \hookrightarrow
    \Big\{
      \sigma \in \Aut(F'/F)
      \,\Big\vert\,
      \sigma(Z)=Z
    \Big\}.
  \]
  Finally, \Cref{lem:extension-of-auts} lets us describe the image of that map, proving \ref{item:general-bijection-aut}:
  \[
    \Out(L/K)
    \simeq
    H \coloneqq
    \left\lbrace
      \sigma \in \Aut(F'/F)
      \,\verti\,
      \begin{matrix}
        \sigma(Z) = Z \\
        \sigma|_{Z} \text{ preserves } L
      \end{matrix}
    \right\rbrace.
  \]
  
  It remains to prove \ref{item:general-bijection-galois}.
  We have seen that the restrictions to $F'$ of automorphisms of $L/K$ are exactly the elements of $H\subseteq\Aut(F'/F)$.
  If $F'/F$ is not a Galois extension or if $H$ is a proper subgroup of $\Aut(F'/F)$, then $F'^H \supsetneq F$.
  But then $L^{\Aut(L/K)} \supseteq (F'\otimes_F K)^{\Aut(L/K)} = F'^H\otimes_F K \supsetneq K$, so $L/K$ is not Galois.
  
  Conversely, if $F'/F$ is Galois extension and every automorphism of $F'/F$ belongs to $H$, then $F'^H = F'^{\Gal(F'/F)} = F$.
  Hence, $(F'\otimes_F K)^{\Aut(L/K)} = F'^H\otimes_F K = F\otimes_F K = K$.
  According to \Cref{lem:inner-galext-incl-center}, the extension $L/F'\otimes_F K$ is Galois, so in particular $L^{\Aut(L/K)} \subseteq L^{\Aut(L/F'\otimes_F K)} = F'\otimes_F K$.
  Together, we conclude that $L^{\Aut(L/K)} = K$, so $L/K$ is Galois.
\end{proof}

\begin{remark}
  Without the assumption that $L$ is a division algebra, the compositum $F'=Z\cdot F$ might not be a field.
  For example, let $L=\mathfrak{M}_2(\Q(i))$ and let $K\subseteq L$ be the $\Q$-algebra generated by the rotation matrix $M=\left(\begin{smallmatrix}0&1\\-1&0\end{smallmatrix}\right)$, whose minimal polynomial is $X^2 + 1$.
  Note that $K$ is a commutative algebra, abstractly isomorphic to $\Q[X]/(X^2+1) \simeq \Q(i)$.
  We have $Z=Z(L)=\Q(i)$ and $F=Z(K)=K$.
  The compositum $F' = Z\cdot F = \Q(i)[M]$ is then isomorphic to $\Q(i)[X]/(X^2 + 1) = \Q(i) \times \Q(i)$ which is not a field.
\end{remark}

\begin{remark}
  In the proof of \Cref{thm:general-bijection}, we constructed an embedding of $F'\otimes_F K$ into $L$.
  Since $F'=Z\cdot F$, the image $A$ of this embedding is the compositum $Z\cdot K$.
  In particular, $Z\cdot K\cong F'\otimes_F K$ is a simple $Z$-algebra.
  By \cite[\href{https://stacks.math.columbia.edu/tag/074T}{Theorem~074T}]{stacks-project}, we have $Z\cdot K = \Cent_L(\Cent_L(Z\cdot K))$.
  Moreover, $\Cent_L(Z\cdot K)=\Cent_L(K)$.
  Thus, the image $A$ can also be constructed as the double centralizer $\Cent_L(\Cent_L(K))$, and the field $F'$ as the center of $A \simeq F'\otimes_F K$.
\end{remark}

In principle, \Cref{thm:general-bijection} suggests an approach for enumerating or counting extensions $L/K$: parametrize number fields $F'/F$ (say, with fixed Galois group $G$), subfields $Z$ of $F'$, and then inner Galois extensions $L/F'\otimes_F K$.
An enumeration of inner Galois extensions $L/F'\otimes_F K$ is obtained by adapting the methods of \Cref{sn:inner} to account for the condition that certain local invariants must coincide
(see \Cref{lem:extension-of-auts} and points \ref{item:general-bijection-aut}, \ref{item:general-bijection-galois} of \Cref{thm:general-bijection}).
The question of the enumeration of number fields $F'/F$ is essentially Malle's conjecture \cite{malle1,malle2}.
However, even in situations where Malle's conjecture has a known answer (for example the case $G = \Z/2\Z$ where $F'/F$ is a quadratic extension), one is left with nontrivial analytic problems (uniformity estimates in constant factors and error terms) that require future research.


\bibliographystyle{alphaurl}
\bibliography{asympskew.bib}
\end{document}